\documentclass[12pt,reqno,twoside]{amsart}
\usepackage{amsmath}
\usepackage{amssymb}
\usepackage[all,cmtip]{xy}
\usepackage{epsfig}
\usepackage{ stmaryrd }
\usepackage{color}
\usepackage{hyperref}
\title{On the Mordell-Gruber spectrum}
\author{Uri Shapira 
  } 
\address{ETH Z\"urich {\tt ushapira@gmail.com}}
\author{Barak Weiss}
\address{Ben Gurion University, Be'er Sheva, Israel 84105
{\tt barakw@math.bgu.ac.il}}
\font\sb = cmbx8 scaled \magstep0
\font\sn = cmssi8 scaled \magstep0

\long\def\combarak#1{\ifdraft{\sb #1 }\else\ignorespaces\fi}

\newif\ifdraft\drafttrue
\draftfalse
\newcommand\name[1]{\label{#1}{\ifdraft{\sn [#1]}\else\ignorespaces\fi}}

\newcommand\eq[2]{{\ifdraft{\ \tt [#1]}\else\ignorespaces\fi}\begin{equation}\label{#1}{#2}\end{equation}}
\newcommand {\equ}[1]{\eqref{#1}}

    \newcommand{\Xn}{{\mathcal{L}_n}}

\newcommand{\goth}[1]{{\mathfrak{#1}}}
\newcommand{\Q}{{\mathbb {Q}}}

\newcommand{\R}{{\mathbb{R}}}

\newcommand{\Z}{{\mathbb{Z}}}

\newcommand{\C}{{\mathbb{C}}}
\newcommand{\BB}{{\mathcal{B}}}
\newcommand{\PPP}{{\mathcal{P}}}

\newcommand{\E}{{\mathbf{e}}}

\newcommand{\wt}[1]{{\widetilde{#1}}}
\newcommand{\pa}[1]{{\left(#1\right)}}
\newcommand{\set}[1]{{\{#1\}}}
\newcommand{\on}[1]{{\operatorname{#1}}}
\newcommand{\Lam}{\Lambda}

\newcommand\mat[1]{\pa{\begin{matrix}#1\end{matrix}}}

\newcommand{\Ad}{{\operatorname{Ad}}}

\newcommand{\GL}{\operatorname{GL}}

\newcommand{\SL}{\operatorname{SL}}

\newcommand{\diag}{{\rm diag}}

\newcommand{\Gal}{{\rm Gal}}
\newcommand {\ignore}[1]  {}

\newcommand{\spa}{{\rm span}}

\newcommand{\conv}{{\rm conv}}

\newcommand{\df}{{\, \stackrel{\mathrm{def}}{=}\, }}

\newcommand{\til}{\widetilde}

\newcommand{\supp}{{\rm supp}}

\newcommand{\sm}{\smallsetminus}

\newcommand{\vre}{\varepsilon}

\newcommand\Vol{\mathrm{Vol}}

\newcommand{\A}{{\mathcal{A}}}

\newtheorem{thm}{Theorem}[section]

\newtheorem{lem}[thm]{Lemma}

\newtheorem{prop}[thm]{Proposition}

\newtheorem{cor}[thm]{Corollary}

\theoremstyle{definition}\newtheorem{remark}[thm]{Remark}

\newtheorem{example}[thm]{Example}

\theoremstyle{definition}\newtheorem{dfn}[thm]{Definition}

\begin{document}
\maketitle

\begin{abstract}
We investigate the Mordell
constant of certain families of lattices, in particular, of lattices
arising from totally real fields. We define the almost sure value
$\kappa_\mu$ of
the Mordell constant with respect to certain homogeneous measures on the
space of lattices, and establish a strict inequality
$\kappa_{\mu_1} < \kappa_{\mu_2}$ when the $\mu_i$ are finite and
$\supp(\mu_1) \varsubsetneq \supp (\mu_2)$. In combination with known results regarding the
dynamics of the diagonal group  we obtain isolation results as well as
information regarding accumulation points of the Mordell-Gruber
spectrum, extending previous work of Gruber and Ramharter. One of the main tools we develop is the associated algebra, an algebraic invariant
attached to the orbit of a lattice under a block group, which can be
used to characterize closed and finite volume orbits. 


\end{abstract}
\section{Introduction}
\subsection{The Mordel constant of a lattice}
Let $\Lambda \subset \R^n$ be a lattice. By a {\em symmetric box} in $\R^n$ we mean a set of the form $ [-a_1, a_1] \times \cdots \times [-a_n, a_n]$, and we say that a symmetric box is
 {\em admissible} if it contains no nonzero points of $\Lambda$ in its interior.  The {\em Mordell constant} of $\Lambda$ is defined to be 
\eq{eq: defn const}{
\kappa(\Lambda) \df \sup_{\BB} \frac{ \Vol(\BB)}{2^n \Vol(\Lambda)},
}
where the supremum is taken over admissible symmetric boxes $\BB$, and where $\Vol(\BB)$ denotes the volume of $\BB$ and $\Vol(\Lambda)$ denotes the volume of a fundamental domain for $\Lambda$. 
The purpose of this paper is to study the quantity $\kappa(\cdot)$, as a function on the space of lattices; in particular, to study its image, which we call the {\em Mordell-Gruber spectrum}, its generic values, and isolation properties. 
Research on these questions stems from the so-called `Mordell inverse problem' \cite{Mordell} and their in-depth study was carried out in a number of papers, notably those of Gruber and Ramharter. We refer to \cite[Chap. 3]{GL} for a detailed history, and give more precise references to the literature below. 

Since the function $\kappa(\Lambda)$ is invariant under homotheties, there is no loss of generality in restricting our attention to unimodular lattices (i.e.\ lattices with $\Vol(\Lambda)=1$). We will denote the space of unimodular lattices of dimension $n$ by $\Xn$. It is equipped with the transitive action of the group $G\df \SL_n(\R)$, and the function $\kappa$ is invariant under the action of the subgroup $A$ of diagonal matrices in $G$ with positive diagonal entries\footnote{For notational convenience we will not work with the full group of linear maps preserving $\kappa$, which besides $A$, also contains non-positive diagonal matrices and permutations of the coordinates.}, since this action permutes symmetric boxes. We will prove new results as well as apply known ones about this $A$-action on $\Xn$ to derive consequences for $\kappa$ --- see \cite{EL} for a survey of the recent progress in the study of this action. 
\subsection{Homogeneous measures and intermediate lattices}
From the dynamical point of view it is natural to study homogeneous
$A$-invariant measures on $\Xn$ which we now define. 
Let $H\Lambda\subset \Xn$ be closed orbit of a real algebraic 
subgroup $H \subset G$. We will see in Proposition~\ref{prop: new closed} that 
the orbit $H\Lam$ supports a locally finite $H$-invariant  measure which is unique up to scaling. 
\begin{dfn}
Given an $A$-invariant closed orbit $H\Lam\subset \Xn$ of a closed
connected real algebraic subgroup $H \subset G$  we refer to the $H$-invariant locally
finite measure supported on $H\Lam$ as the \textit{homogeneous
  measure} associated with the orbit $H\Lam$ and denote it by
$\mu_{H\Lam}$. The closed orbit $H\Lam$ will be referred to as the
\textit{homogeneous space} corresponding to the measure. 
\end{dfn}
We emphasize that we allow our homogeneous spaces to be of infinite
measure; when the measure $\mu_{H\Lam}$ is finite  we say that the
orbit $H\Lam$ is \textit{of finite volume}. It is well known that the
orbit $G\Lam=\Xn$ is of finite volume and  we denote the  
corresponding (unique) $G$-invariant probability measure by $\mu_{\Xn}$.
The starting point of our discussion is the following
\begin{thm}\name{u thm1}
Let $\mu$ be a homogeneous $A$-invariant measure on $\Xn$. Then for $\mu$-almost any $\Lambda$ 
\begin{equation}\label{u thm1 eq}
\kappa(\Lambda)=\max\{\kappa(\Lambda'):\Lambda'\textrm{ is in the support of }\mu\}.
\end{equation}
\end{thm} 
Theorem \ref{u thm1} is a standard consequence of the ergodicity of
the $A$-action and is proved in \S\ref{sec: ergodicity}.  
We note that 
a well-known conjecture of Margulis \cite{Margulis conjectures}
asserts that in dimension $n\ge 3$ any $A$-invariant and $A$-ergodic
probability measure on $\Xn$ is homogeneous.

The following consequence of Theorem~\ref{u thm1} answers 
a question of Gruber, and improves on previous results of Gruber and
Ramharter \cite{GR1, R1, R2}. Note that Minkowski's convex body
Theorem implies that $\kappa(\Lam) \leq 1$ for any $\Lam$, this upper
bound being attained by $\Lam = \Z^n$. Therefore taking
$\mu=\mu_{\Xn}$ we obtain:
\begin{cor}\name{thm: full measure}
With respect to $\mu_{\Xn}$, almost every lattice has Mordell constant equal to 1. 
\end{cor}
A natural question is the existence and characterization of lattices with Mordell constant strictly smaller than 1.
 In view of Theorem~\ref{u thm1}, given a homogeneous $A$-invariant
 measure $\mu$,  it makes sense to define \textit{the generic value}
 $\kappa_\mu$  
 to be the almost sure value of $\kappa$ with  respect to $\mu$. One
 of the main results of this paper is the following: 
\begin{thm}\label{u thm 2}
Let $\mu_1,\mu_2$ be two $A$-invariant homogeneous 
measures such that $\mu_1$ is finite and 
$\operatorname{supp}(\mu_1)\varsubsetneq\operatorname{supp}(\mu_2)$. Then  
$\kappa_{\mu_1}< \kappa_{\mu_2}.$
\end{thm}
In \S\ref{sec: the main theorem} we show by examples that the hypothesis that $\mu_1$
is finite in Theorem \ref{u thm 2} is essential. 
Nevertheless,  we will establish
Theorem~\ref{main theorem} which extends Theorem~\ref{u thm 2} to
the case of $A$-invariant homogeneous measures which are not
necessarily finite, under a suitable additional assumption. 

In order to prove Theorems~\ref{u thm 2} and \ref{main theorem} we
will study homogeneous $A$-invariant 
measures. 
As will be shown in Proposition \ref{prop: why block groups}, the
groups $H$ that give rise to homogeneous $A$-invariant measures are
\textit{block groups} obtained by 
choosing a partition $\PPP=\sqcup_1^r Q_\ell$ of $\set{1\dots n}$ and
defining
\begin{equation}\label{bgps}
H(\PPP)=\set{(g_{ij})\in G: g_{ij}\ne0\Rightarrow i,j\in Q_\ell\textrm{ for some }\ell}^\circ
\end{equation}
(where $L^\circ$ is the connected component of the identity in the
group $L$). 
In~\S\ref{subsection: types} we study orbits of block groups in
detail.
We attach to each orbit $H\Lam$ of a block group an algebraic
invariant we refer to as the \textit{associated algebra} 
which is a finite dimensional $\Q$-algebra.  Simple algebraic
properties of the associated algebra allow us to determine whether the
orbit is closed or of finite volume  (see
Theorem~\ref{n.c.c}). Whenever we have a containment $H_1\Lam\subset
H_2\Lam$ of orbits as above, we have a reverse inclusion of the
associated algebras and the condition which allows us to generalize
Theorem~\ref{u thm 2} is a simple algebraic property of the inclusion
of the associated algebras. 

\begin{dfn}\label{i.l.d.1}
A lattice $\Lambda\in\Xn$ is said to be {\em intermediate} (resp.\
\textit{intermediate of  finite volume type}) if it belongs to an
$A$-invariant homogeneous space
(resp.\ of finite volume)  $H\Lambda$   which is strictly contained in $\Xn$. 
\end{dfn}
The case $H=A$ will be of particular interest.
\begin{dfn}\label{i.l.d.2}
A lattice $\Lambda$ for which $A\Lambda$ is closed will be referred to
as an \textit{algebra lattice}. If furthermore the orbit is of finite
volume (equivalently, if it is compact)  then  the lattice is said to
be a {\em number 
field lattice}.
\end{dfn}
Corollary~\ref{cor: compact A orbits} will justify our choice of
terminology. 
It shows that  a lattice $\Lam$ is an
algebra  (resp.\ number field) lattice  
if and only if the associated algebra is $n$-dimensional (resp. is an $n$-dimensional field) over the rationals.

Combining  Theorems \ref{u thm1} and \ref{u thm 2} we obtain the following 
\begin{cor}\label{u cor1}
If $\Lambda\in\Xn$ be an intermediate lattice of finite volume type,
then $\kappa(\Lambda)<1$. 
\end{cor}
Corollary \ref{u cor1} is probably not new for number field lattices
 (see e.g. \cite{R1}) but we could not locate a suitable reference.
 \subsection{Isolation results}
  The results below concern isolation properties that follow  from 
  Theorem~\ref{u thm 2} and a rigidity result for the $A$-action in
  dimension $n\ge 3$ (Theorem~\ref{thm: dynamical isolation}). 
\begin{dfn}\name{def: isolated}
Let $\Lambda$ be a lattice and let $\vre_0>0$. We say that $\Lambda$ is {\em $\vre_0$-isolated} if 
for any $0< \vre <\vre_0$ there is a neighborhood $\mathcal{U}$ of $\Lambda$ in $\Xn$, such that for any $\Lambda' \in \mathcal{U} \sm A\Lambda$,  $\kappa(\Lambda') > \kappa(\Lambda)+\vre$. 

We say that $\Lambda$ is {\em locally isolated} if it is
$\vre_0$-isolated for some $\vre_0>0$, and that $\Lambda$ is {\em
  strongly isolated} if it is $\vre_0$-isolated for $\vre_0 =
1-\kappa(\Lambda)$.  
 \end{dfn}
   
\begin{thm}\name{thm: isolation nf} 
Let $n \geq 3$, and let $\Lambda$ be a number field lattice, 
associated with the degree $n$
number field $F$. Then $\Lambda$ is locally
isolated. Moreover $\Lambda$ is strongly isolated if and only if there
are no intermediate fields $\Q \varsubsetneq K \varsubsetneq F$.

\end{thm}

Theorem \ref{thm: isolation nf} 
extends results of Ramharter \cite{R1}, who shows local isolation
under an additional assumption, but does not require $n \geq 3$. Our
methods crucially rely on the hypothesis $n \geq 3$. An immediate
consequence is: 

\begin{cor}\name{cor: prime nf}
If $n \geq 3$ is prime, then any 
number field 
lattice  in $\R^n$ is
strongly isolated.  
\end{cor}

We remark that when $n$ is prime, an intermediate lattice of finite
volume type is automatically a number field lattice. 
Extending another result of \cite{R1} we show:
\begin{cor}\name{cor: extends}
For any $n \geq 3$, the set of strongly isolated lattices is dense in $\Xn$. 
\end{cor}

 The fundamental difference between the cases $n=2$ and $n \geq 3$ is highlighted in the following:

\begin{thm}\name{thm: n=2}
In dimension $n=2$ there are no strongly isolated lattices. 
\end{thm}

Theorem \ref{thm: n=2} relies on work of Gruber \cite{Gruber}. 
In contrast with Theorem \ref{thm: isolation nf}, intermediate 
lattices which are not number field lattices are typically not isolated. 
In \S \ref{sec: completion} we will define a notion of `local relative isolation' and prove: 
\begin{thm}\name{thm: relative isolation}
Let  $\mu$ be an $A$-invariant homgeneous probability measure corresponding to a finite volume orbit $H\Lambda$ with $A\varsubsetneq H\varsubsetneq G$.
Then 
almost any lattice with respect to $\mu$ is not locally isolated but is locally isolated relative to $H$.     
\end{thm}

\subsection{The reduced Mordell-Gruber spectrum}
We denote by $\mathbf{MG}_n$ the {\em Mordell-Gruber spectrum}, which
is the set of numbers $\kappa(\Lambda)$ where $\Lambda$ ranges
over all lattices in dimension $n$. We briefly summarize some of the known facts
about the Mordell-Gruber spectrum. Siegel (see \cite{GL}) showed that
$\kappa_n \df \inf \mathbf{MG}_n>0$. Many things are known about $\mathbf{MG}_2$, see \cite{Gruber}. The values $\kappa_2$ and
$\kappa_3$ are known, the latter by a difficult work of Ramharter
\cite{Ramharter 3}.  In \cite{Ramharter 3} Ramharter also showed that
$\kappa_3$ belongs to $\mathbf{MG}_3$ and is an isolated\footnote{The
  isolation of a number in a subset of $\R$ should not be confused
  with the isolation property of Definition \ref{def: isolated}.}
point. 
Various lower
bounds on $\kappa_n$ have been proved by various authors, and recently
\cite{stable} the authors obtained the lower bound $\kappa_n \geq
n^{-n/2}$.  

We wish to study accumulation points of $\mathbf{MG}_n$. There is a simple trick to generate such points which we now describe. 
As we explain in \S \ref{sec: n=2}, it can be deduced from results of Gruber that $\mathbf{MG}_2$ has many accumulation points. 
We say that a lattice $\Lambda \subset \R^n$ is  {\em decomposable} if
$n=n_1+n_2, \, n_i>0,$ and $\R^n = \R^{n_1} \oplus \R^{n_2}$ is the
direct sum decomposition corresponding to partitioning the coordinates
into subsets of sizes $n_1$ and $n_2$, and we can write $\Lambda =
\Lambda_1 \oplus \Lambda_2$, where $\Lambda_i = \Lambda \cap
\R^{n_i}.$  In this case we clearly have $\kappa(\Lambda) =
\kappa(\Lambda_1)\kappa(\Lambda_2)$. Taking direct sums with
$\Z^{n_2}$ we get embeddings $\mathbf{MG}_{n_1} \hookrightarrow
\mathbf{MG}_n$ for any $n_1 < n$. We are interested in the part of the
spectrum not arising in this way. That is, we define the {\em reduced
  Mordell-Gruber spectrum} to be 
$$
\widehat{\mathbf{MG}}_n \df \{\kappa(\Lambda): \Lambda \subset \R^n \mathrm{\ a \ lattice \ which \ is \ not \ decomposable } \}.
$$
As will be seen in Proposition \ref{prop: never decomposable}, number field lattices are never
decomposable. 
We are interested in the existence of accumulation points of $\widehat{\mathbf{MG}}_n$.   
\begin{thm}\name{thm: non-isolated, intermediate}
Let  $\mu$ be an $A$-invariant homogeneous probability measure
corresponding to a finite volume orbit $H\Lambda$ with $A\varsubsetneq
H$.  Then there is a  
sequence of number field lattices $\Lambda_k$ 
for which $\kappa(\Lambda_k) \nearrow \kappa_\mu.$ In particular, $\kappa_\mu$ is not an isolated point of $\widehat{\mathbf{MG}}_n$.
\end{thm}
Taking $\mu$ to be the Haar measure we obtain
\begin{cor}\name{thm: non-isolated}
For any $n$ there is a sequence 
 $(\Lambda_k)$
of number field lattices 
for which $\kappa(\Lambda_k) \nearrow 1.$ In particular 1 is not an isolated point of $\widehat{\mathbf{MG}}_n$. 

\end{cor}

Given $A \subset \R$, we denote $A^{(0)} \df A$, and by $A^{(k+1)}$ the elements of $A^{(k)}$ which are limits of strictly increasing sequences in $A^{(k)}$. 

\begin{thm}\name{thm: hopefully2}
For any natural number $t$, there is $n$ so that  $1 \in \widehat{\mathbf{MG}}_n^{(t)}$. 
\end{thm}

\subsection{Organization of the paper}
In sections  \S\ref{sec: prelims} and \S\ref{sec: ergodicity} we recall
some standard results and prove some useful results about closed orbits
for actions of algebraic groups on $\Xn$. From these we deduce Theorem~\ref{u thm1}.
In \S\ref{subsection: types} we introduce  the \textit{associated
  algebra} of a lattice and characterize intermediate lattices in
terms of its algebraic properties. As we explain in \S \ref{sec:
  consequences}, the
associated algebra of a lattice $\Lam$ can be used to classify all
$A$-invariant homogeneous subsets containing $\Lam$. Moreover 
in \S \ref{subsec: constructions}
we show how to explicitly construct all intermediate lattices. The
proof of Theorem~\ref{u thm 2} is given 
in~\S\ref{sec: the main theorem} and of the isolation results in~\S\ref{sec:
  completion}.  In \S \ref{sec: n=2} we recall results of Gruber and
Ramharter for dimension $n=2$, give some more information about
$\mathbf{MG}_2$,  and prove Theorem \ref{thm: n=2}.  

\subsection{Caveat to the reader}
Some  of the results of this paper rely on results proved in
\cite{LW}, but not stated explicitly there. We were faced with the
dilemma of trying to keep the paper self-contained without repeating
arguments given in \cite{LW}. 
We chose not to include a proof of Theorem \ref{thm: dynamical
  isolation}, which is required for many of our results, was not
stated explicitly in \cite{LW} but follows from the proofs of the main
results (Theorems 1.1 and 1.3) of that paper. To our defense we should
mention that Theorem \ref{thm: dynamical isolation} has been used, and
a proof-sketch given, in \cite{ELMV}, and that a completely
self-contained proof would require several pages of arguments lifted
almost verbatim from \cite{LW}.

\subsection{Acknowledgements}
This paper was conceived when the authors were visiting the Erwin
Schr\"odinger Institute in Vienna in October 2011, as part of the
program {\em Combinatorics, Number theory, and Dynamical Systems}. The
support of ESI is gratefully acknowledged.   
 We are grateful to G. Ramharter and P. A. Gruber for directing our
 attention to the questions discussed in this paper, and for useful
 discussions. We are grateful to Ido Efrat, Alex Gorodnik and Dmitry
 Kleinbock for useful discussions and pointers to the literature.  
We gratefully acknowledge support of European Research Council grants
DALGAPS 279893 and  Advanced research Grant 228304,  and  Israel Science Foundation grant 190/08.

\section{Generalities}\name{sec: prelims}
The following two propositions are standard and explained in \cite{GL}. 
\begin{prop}\name{prop: standard}
The function $\kappa: \Xn \to \R$ has the following properties:
\begin{itemize}
\item[(i)]
For all $\Lambda$, $\kappa(\Lambda) \leq 1$. 

\item[(ii)]
If $\Lambda_k \to \Lambda$ then $\kappa(\Lambda) \leq \liminf \kappa (\Lambda_k);$
i.e.\ 
$\kappa$ is lower semi-continuous.

\item[(iii)]
For all $\Lambda$ and all $a \in A$, $\kappa(a\Lambda)=\kappa(\Lambda).$ 
\item[(iv)] $\kappa(\Z^n)=1$. 
\end{itemize}

\end{prop}

\begin{prop}[Mahler's compactness criterion]\name{prop: Mahler}
A subset $X\subset \Xn$ is bounded (i.e.\ has compact closure) if and only if there is a neighborhood $\mathcal{U}$ of $0$ in $\R^n$ such that for any $\Lambda \in X$, $\Lambda \cap \mathcal{U} = \{0\}$.  
\end{prop}

\begin{cor}\name{cor: cubes}
If $\Lambda \in \Xn$ and $\BB_0$ is a symmetric cube whose volume is smaller than $2^n\kappa(\Lambda)$,  then there is $a \in A$ such that $\BB_0$ is admissible for $a\Lambda$. In particular 
for any $\kappa_0>0$ there is a compact $K \subset \Xn$ such that for any $\Lambda \in \Xn$ with $\kappa(\Lambda)\geq \kappa_0$, there is $a \in A$ such that $a\Lambda \in K$. 
\end{cor}
\begin{proof}
 For the first assertion, let $\BB$ be an admissible symmetric box such that $\Vol(\BB) > \Vol(\BB_0)$, and let $a \in A$ such that $a\BB$ is a cube symmetric about the origin. By considering volumes we see that $\BB_0 \subset a\BB$. This proves the first assertion. The second assertion follows via Proposition \ref{prop: Mahler}. 
\end{proof}

\subsection{Algebraic groups and $\Q$-structures} 
In the remainder of this section we will recall several classical results
about algebraic groups and lattices in Lie groups. We refer the reader
to \cite{Raghunathan} for more details and pointers to the literature. 
In this paper, we have preferred a concrete point of view so we will
work throughout with subgroups of $G = 
\SL_n(\R)$ and with the space $\Xn \cong \SL_{n}(\R) /\SL_n(\Z)$, rather than the more general
setup where $G$ is a real algebraic group and $G/\Gamma$ is the
quotient of $G$ by a lattice $\Gamma$. All the results we state below
are valid in this more general context. 

Given a lattice $\Lam\in\Xn$ we denote $V_\Lambda\df
\on{span}_{\Q}\Lam$. Note that $V_\Lambda$ is a $\Q$-vector subspace
of $\R^n$ such that $V_\Lambda \otimes_{\Q} \R = \R^n$, but
$V_\Lambda$ need not coincide with the standard $\Q$-structure
$\Q^n$. 
We say that a matrix $g\in G$ is $\Lam$-rational
if $gV_{\Lam}\subset V_{\Lam}$; the reader may verify that $g \in
G(\Q)$ if and only if $g$ is $\Z^n$-rational. As in
\cite[Preliminaries, \S2]{Raghunathan} one uses a $\Q$-structure on
$\R^n$ to define 
$\Q$-algebraic subgroups of $\SL_n(\R)$. If we use the $\Q$-structure
of $V_\Lambda$ we will say that such a subgroup $H$ is \textit{defined 
  over $\Q$ with respect to the $\Q$-structure induced by
  $\Lam$}. We will use two characterizations of such subgroups. They
are the groups 
$H$ containing a Zariski dense set of $\Lam$-rational
points; they are also the groups $H$ such that for any $g \in G$ for which $g\Z^n = \Lambda$,
the conjugate $g^{-1}Hg$ is a $\Q$-subgroup of $G$ (with respect to
the standard $\Q$-structure $\Q^n$). See \cite{Raghunathan} for
definitions of morphisms defined over $\Q$ and $\Q$-characters. 

We recall the following classical fact. 
\begin{prop}[Borel Harish-Chandra]\name{thm: BHC} 
Let $H\subset G$ be an algebraic group defined over $\Q$ with respect
to the $\Q$-structure induced by $\Lam\in\Xn$. Then the orbit
$H\Lam\subset \Xn$ is of finite volume if and only if $H$ has no
non-trivial $\Q$-characters. In particular, if $H$ is semisimple or is
generated by unipotent elements then $H\Lam$ is of finite volume. 
\end{prop}

\subsection{Closed orbits of real algebraic subgroups}
The following observation is useful.
\begin{prop}\name{prop: Zcl contains unips}
Let $\Lambda \in \Xn$ and let $H$ be an algebraic subgroup of
$G$. Denote by $H_{\Lambda}$ the stabilizer of $\Lambda$ in $H$ and by
$H_0$ the Zariski closure
of $H_{\Lambda}$. Then $H_{\Lambda}$ is a
lattice in $H_0$. Moreover, if the orbit $H\Lambda$ is closed in $\Xn$
then the connected component $H^\circ_0$ of the identity in $H_0$
contains the unipotent elements of $H$. 
\end{prop}

\begin{proof}
We use the $\Q$-structure induced by $\Lambda$. Any $\Q$-character
$\chi$ of an algebraic group containing $H_{\Lambda}$ 
is bounded below by a bounded denominators argument. Therefore
$\chi(H_{\Lambda})$ is finite, which implies that $\chi|_{H^{\circ}_0}$ is
trivial. By Proposition \ref{thm: BHC}, $(H_0)_{\Lambda}$ is a lattice
in $H_0$, but since $H_{\Lambda} \subset H_0 \subset H$ we have
$(H_0)_{\Lambda} = H_{\Lambda}$. 

Now suppose $u$ is a unipotent element of $H$ and suppose $H\Lambda$
is closed. The orbit $\{u^n \Lambda: n =1,2, \ldots\}$ is
non-divergent in $\Xn$ by a classical result of Margulis \cite{Margulis
  nondivergence}. The orbit
map $hH_{\Lambda} \mapsto h\Lambda$ is proper since we have assumed
that $H\Lambda$ is closed, and this implies that the orbit $\{u^n
H_{\Lambda} : n=1,2, \ldots\}$ is non-divergent in the quotient
$H/H_{\Lambda}$. We have an $H$-equivariant factor map $H/H_{\Lambda} \to
H/H_0$, so the orbit $\{u^nH_0: n=1,2, \ldots\}$ is non-divergent in
$H/H_0$, which is an algebraic variety on which a unipotent trajectory
is either a fixed point or is divergent. This implies that $u \in
H_0$, and since every unipotent $u$ can be embedded in a one-parameter
unipotent subgroup, we have $u \in H^\circ_0$. 
\end{proof}

As we allow homogeneous subspaces of infinite measure, we need the
following fact:
\begin{prop}\label{prop: new closed}
Let $H \subset G$ be a real algebraic group and $\Lambda \in \Xn$
such that $H\Lambda$ is closed. Then there is an $H$-invariant
locally finite measure on $H\Lambda$. This measure is unique up to
scaling.  
\end{prop}
We remark that in contrast to finite volume homogeneous spaces, for
infinite homogeneous spaces $H$ need not be unimodular. 
We also remark that in the statement of the Proposition one may
replace $\Xn$ with any homogeneous space $G/\Gamma$ where $G$ is a real algebraic group
and $\Gamma$ is a lattice. 

\begin{proof}
Let $H_{\Lambda}$ denote the stabilizer of $\Lambda$ in $H$. There is an injective {\em orbit map} 
$$
H/H_{\Lambda} \to \Xn, \ \ \ hH_{\Lambda} \mapsto h\Lambda, 
$$
and the assumption that $H\Lambda$ is closed implies that this map is
a homeomorphism onto its image. 
So it is enough to prove that there is an $H$-invariant
locally finite measure on $H/H_{\Lambda}$. For a Lie group $L$, let
$\Delta_L$ denote its modular function. 
In light of general facts about quotients of Lie groups (see
e.g. \cite[Chapter 1]{Raghunathan}), it is enough to show that
\eq{eq: need to show now}{\Delta_H|_{H_{\Lambda}} = \Delta_{H_{\Lambda}}.}
Since $H_{\Lambda}$ is
discrete, it is unimodular, so $\Delta_{H_{\Lambda}}$ is trivial. So we
need to show that the restriction of  $\Delta_H$ to $H_{\Lambda}$ is
trivial. We have the explicit formula 
$$\Delta_H(h) = |\det \Ad(h)|_{\goth{u}}|,$$ 
where $\goth{u}$ is the Lie algebra of the unipotent radical $U$ of
$H$. In other words $\Delta_H(h)$ is the multiplicative factor by which conjugation by $h$
multiplies the Haar measure on $U$. 

We will now show that $\Gamma \df U \cap H_{\Lambda}$ is a lattice in $U$. 
Indeed, 
let $H_0$ be as
in Proposition \ref{prop: Zcl contains unips}, let $H'$ be the
subgroup of $H$ generated by the unipotent elements of $H$, and let
$U_0, \, U'$ denote respectively the unipotent radicals of $H_0, \,
H'$.
By Proposition \ref{prop: Zcl contains unips}, 
$$H' \subset H_0 \subset H,$$ 
which implies 
$$ U \subset U_0 \subset U'.$$
On the other hand, $H'$ is a characteristic subgroup of $H$, and hence
so is its unipotent radical. Therefore $U'$ is normal in $H$,
which implies that $U=U'$. In particular $U=U_0$. It now follows from
\cite[Cor. 8.28]{Raghunathan} that $\Gamma \cap U$ is a lattice in $U$, as
required. 
\ignore{

 Clearly $U \subset U_0$, and
we claim that $U=U_0$. Indeed, let $H'$ be the subgroup of $H$
generated by the unipotent elements of $H$. Then $H' \subset H_0$ by
Proposition \ref{prop: Zcl contains unips}. The unipotent radical $U'$
of
$H'$ also contains $U$

 $R$ be the
radical of $H_0$, i.e. the largest connected normal solvable subgroup
of $H_0$. 
Then by Auslander's Theorem \cite[Chap. 8]{Raghunathan}, $R
\cap H_{\Lambda}$ is a lattice in $R$. 
Clearly $U \subset R$, and since $R$ is characteristic, the subgroup
of $R$ generated by unipotent elements of $R$ is normal in $H_0$,
i.e. it coincides with $U$. 
Therefore by a theorem of
Mostow \cite[Chap. 3]{Raghunathan}, $\Gamma$ is a lattice
in $U$. 
}
Since
conjugation by elements of $H_{\Lambda}$ preserves both $U$ and
$H_{\Lambda}$, it fixes $\Gamma$ and so fixes the covolume of
$U/\Gamma$. This implies \equ{eq: need to show now}.
\end{proof}


\section{Ergodicity and consequences} 
\name{sec: ergodicity}
In this section we prove Theorem~\ref{u thm1} by establishing the ergodicity of the $A$-action with respect to homogeneous $A$-invariant measures (see Proposition~\ref{prop: ergodic}). In order to establish this, we study in some detail the structure of homogeneous $A$-invariant spaces in $\Xn$.

Let $X$ be a locally compact topological space, let $\mu$ be a locally finite Borel measure on $X$, and let $A$
be a group acting continuously on $X$ preserving 
$\mu$. The action is called {\em ergodic} if any invariant set is
either of zero measure, or its complement is of zero measure. Let $\on{supp}( \mu)$ denote the topological support of $\mu$. Then it is well-known (see  e.g. \cite{Zimmer}) that when the action is ergodic, any $A$-invariant measurable function $X \to \R$ is almost everywhere constant and for almost every $x \in X$, the orbit $Ax$ is dense in $\on{supp}(\mu)$.  
\begin{thm}\name{thm: ergodic+semicontinuous}
Let $X = \overline{A\Lambda_0}$ be an orbit-closure for the $A$-action on $\Xn$. Then 
$$\kappa(\Lambda_0) = \sup \{\kappa(\Lambda): \Lambda \in X\}.$$
In particular, if $\mu$ is an $A$-invariant and $A$-ergodic measure, then for $\mu$-almost every $\Lambda_0$ we have
\eq{eq: kappa}{
\kappa_\mu\df \sup\{\kappa(\Lambda): \Lambda \in \on{supp}(\mu)\} = \kappa(\Lambda_0)
}
and so the supremum in \equ{eq: kappa} is attained. Moreover
$\on{supp}(\mu)$ contains a lattice $\Lambda_{\max}$ with
$\kappa(\Lambda_{\max}) = \kappa_\mu$ such that the cube $\mathcal{C}$
of volume $2^n \kappa_\mu$ is admissible for $\Lambda_{\max}$ (so the
supremum in \equ{eq: defn const} is attained for $\Lambda =
\Lambda_{\max}, \, \BB=\mathcal{C}$).  
\end{thm}
\begin{proof}
By Proposition \ref{prop: standard}(ii),(iii), for any $\Lambda \in
X$, $\kappa(\Lambda) \leq \kappa(\Lambda_0)$. This proves the first
assertion. The second one follows taking $X = \on{supp}(\mu)$ and
recalling that almost every $A$-orbit in $X$ is dense.   

For the last assertion, let $\Lambda \in \on{supp}(\mu)$ with
$\kappa(\Lambda) = \kappa_\mu$ and let $C_k$ be a sequence of
symmetric cubes with $\Vol(C_k) \nearrow 2^n \kappa(\Lambda)$. For
each $k$, by Corollary \ref{cor: cubes} there is $a_k \in A$ so that
$C_k$ is admissible for $a_k(\Lambda)$ and the sequence
$\{a_k\Lambda\}$ is contained in a bounded subset of $\Xn$. Let
$\Lambda_{\max}$ be a limit of a converging subsequence of $\{a_k
\Lambda\}$. Then $\Lambda_{\max} \in \on{supp}(\mu)$ since
$\on{supp}(\mu)$ is closed and $A$-invariant. Moreover by
construction, the cube of volume $2^n \kappa_\mu$ is admissible for
$\Lambda_{\max}$. This implies that $\Lambda_{\max}$ has the required
properties.  
\end{proof}

\subsection{Block groups}
Given a partition of the indices $\{1,\dots, n\}$
\eq{eq: partition0}{
\PPP \df \left( \{1, \ldots, n\} = \bigsqcup Q_{\ell} \right) ,
}
we define the  {\em block group} corresponding to $\PPP$ to be the connected subgroup $H=H(\PPP)$ of $G$ whose Lie algebra is 
\eq{eq: equiblock group}{
\goth{h} = \goth{a} \oplus \bigoplus_{\ell} \bigoplus_{s,t \in Q_{\ell}} \goth{g}_{st}
}
where $\goth{a}$ are the Lie algebra of $A$ and $\goth{g}_{st}$ is the
one-dimensional Lie algebra spanned by the matrix with 1 in the entry
$(s,t)$ and 0 elsewhere (note that we always have $A \subset H$). We
refer to the elements $Q_\ell$ of $\PPP$ as the \textit{blocks} of the
partition and denote by $|\PPP|$ the number of blocks. When the blocks
are of equal size we say that $\PPP$ is an \textit{equiblock
  partition} and $H(\PPP)$ is an \textit{equiblock group}.  
Given a partition $\PPP$ we shall denote by $\sim_\PPP$ the
equivalence relation it defines on $\set{1, \dots, n}$.  
Our interest in block groups is explained by the following: 
\begin{prop}\name{prop: why block groups}
Let $H\Lambda\subset\Xn$ be a homogeneous space (i.e.\ a closed orbit
of a closed connected subgroup  $H\subset G$). Then if $H\Lambda$ is
$A$-invariant, then $A\subset H$ and  
$H=H(\PPP)$ for some partition $\PPP$.
\end{prop}
We will use the following simple Lemma whose proof is left to the reader.
\begin{lem}\label{containment of groups}
If $H_1\Lam\subset H_2\Lam$ is a containment of two orbits in $\Xn$ of
closed groups $H_1,H_2$ and $H_1$ is connected, then $H_1\subset
H_2$. 
\end{lem}
\qed
\begin{proof}[Proof of Proposition \ref{prop: why block groups}]
The fact that $A\subset H$ follows from Lemma~\ref{containment of groups}.
Note that if $H\subset G$ is a closed connected subgroup containing
the diagonal group then in the above notation, 
 the Lie algebra of $H$ satisfies 
$\goth{h} = \goth{a} \oplus \bigoplus \goth{g}_{st}$, where the sum is
taken over some subset of the set of pairs $(s,t)$. Since $\goth{h}$
is a Lie algebra, $\goth{g}_{s,t}, \goth{g}_{t,u} \subset \goth{h}$
implies $\goth{g}_{s,u} \subset \goth{h}$. We need to show that 
$$
\goth{g}_{s,t} \subset \goth{h} \ \implies \ \goth{g}_{t,s} \subset \goth{h}.
$$
Let $u_{st} \subset H$ be the one parameter unipotent group with
Lie algebra   
$\goth{g}_{st}$. There exists a one-parameter subgroup $A_0 \subset A$,
such that the group $B$ generated by
$A_0, \, \{u_{ts}(x)\}$ is the Borel subgroup of the copy of
$\SL_2(\R)\subset G$ which is generated by the two groups
$u_{ts}(x),u_{st}(x)$. We denote this copy of $\SL_2(\R)$ by $H_0$.  
As $H\Lam$ is assumed to be closed we have 
$H\Lam \supset \overline{B\Lam}$. By the work
of Ratner (see \cite{Ratner - padic} for a short proof) we have
$\overline{H_0\Lam}= \overline{B\Lam}$ and so we conclude that
$\set{u_{st}(x)\Lam:x\in\R}\subset H\Lam$. 
By Lemma~\ref{containment of groups}, $\set{u_{st}(x)}\subset H$ as desired.
\end{proof}
We will see in Corollary~\ref{equi must} that in the case $H\Lam$ is
of finite volume, there are further restrictions on the partition
$\PPP$ in the above proposition.

\subsection{Structure of $A$-invariant homogeneous measures}
Let $H=H(\PPP)$ be a block group. Then it can be written as an almost direct
product $H(\PPP)=Z(\PPP)\cdot S(\PPP)$, where $Z(\PPP)\subset A$ is
the  center of $H(\PPP);$ that is, $Z(\PPP)$ contains those diagonal
matrices in $H(\PPP)$ that have constant eigenvalues along the blocks
of $\PPP$, and  
 $S(\PPP)$ is the commutator group of $H(\PPP)$. More concretely,
 $S(\PPP)$ is the semisimple group of matrices having the block
 structure given by $\PPP$ with the further requirement that the
 determinant of each block is 1. 
 
 The following proposition shows that an $A$-invariant homogeneous
 measure has a simple product structure. 
\begin{prop}\name{prop: product structure}
Let $H=H(\PPP)$ be a block group, $H\Lambda$ an $A$-invariant
homogeneous space, and $\mu=\mu_{H\Lambda}$ the corresponding
$A$-invariant measure.  
Then there is a decomposition 
of $Z=Z(\PPP)$ as a direct product $Z=Z_s\cdot Z_a$ such that 
\begin{enumerate}
\item\label{deco.2} If $H_1=Z_a\cdot S$, where $S=S(\PPP)$, then 
  $H_1\Lam$ is of finite volume. 
\item\label{deco.1}  The map 
$$Z_s\to Z_s\Lam, \ \ z\mapsto z\Lam$$
  is proper, and a homeomorphism onto its image. In particular, the orbit $Z_s\Lam$ is
  divergent. 
\item\label{deco.3} The map 
$$Z_s\times H_1\Lam \to
  H\Lam, \ \ (z,h\Lam)\mapsto zh\Lam$$ 
is a homeomorphism onto its image, under which $\mu$ is identified
with 
 $\nu\times \mu_{H_1\Lam}$,
where  $\nu$ is Haar measure on $Z_s$. 
\end{enumerate}
\end{prop}
\begin{remark}
Note that, since $Z$ is central in $H$, the conclusions~\eqref{deco.2} and
\eqref{deco.1} hold for any $\Lam'\in H\Lam$. 
\end{remark}
\begin{proof}
Let $H_1\subset H$ denote the connected component of the identity in
the Zariski closure of the stabilizer group 
$H_\Lambda$. Replacing $H_{\Lambda}$ with a finite-index subgroup if necessary, we
may assume that $H_{\Lambda} \subset H_1$. It follows from Proposition \ref{prop: Zcl contains unips} 
that 
the orbit $H_1\Lam \cong H_1/H_{\Lambda}$ is of finite 
volume, and also that 
$S\subset H_1$. Since 
$$S\subset H_1\subset Z\cdot S=H$$ 
we find that 
 $H_1=Z_a\cdot S$, where $Z_a\df H_1\cap Z$ which
 establishes~\eqref{deco.2}. 

Let $Z_s$ be any direct complement of $Z_a$ in $Z$; that is, a
subgroup of $Z$ such that  $Z=Z_s\cdot Z_a$ (a direct product).  
Consider the natural embeddings
\begin{equation}\label{n.emb}
H_1/H_\Lam\hookrightarrow H/H_\Lam\hookrightarrow H\Lam\subset \Xn
\end{equation}
and note that as $H=Z_s\cdot H_1$ and $H_\Lam\subset H_1$,  the space
$H/H_{\Lam}$ is naturally identified with $Z_s\times H_1/H_\Lam$. 
As the orbits $H\Lam, H_1\Lam$
are closed, the embeddings in~\eqref{n.emb} are proper.  Moreover our
choice of $Z_s$ implies that the stabilizer of $\Lambda$ in $Z_s$ is
trivial. 
This implies 
\eqref{deco.1} and the first statement
of~\eqref{deco.3}. The statement regarding the measures now follows
from the uniqueness of an $H$-invariant measure
\cite[Chap. 1]{Raghunathan} on $H/H_{\Lambda}$. 
\end{proof}

\begin{prop}\name{prop: ergodic}
Let $\mu$ be an $A$-invariant homogeneous measure on $\Xn$
corresponding to the closed orbit $H\Lambda$. Then the $A$-action is
ergodic with respect to $\mu$. 
\end{prop}
\begin{proof}
We use the notation of Proposition~\ref{prop: product structure}.
Identifying the orbit $H\Lam$  with the product $Z_s\times H_1\Lam$ we
see that, since $Z_s\subset A$, the statement reduces to the
ergodicity of the action of 
$A\cap H_1$ with respect to the finite
$H_1$-invariant measure $\mu_{H_1\Lam}$. The latter statement follows
from the 
Howe-Moore Theorem (see e.g.\ \cite{Zimmer}). 
\end{proof}
\begin{proof}[Proof of Theorem \ref{u thm1}]
The statement follows from Theorem \ref{thm: ergodic+semicontinuous}
and Proposition \ref{prop: ergodic}. 
\end{proof}

  \begin{remark}\label{anisotropic torus}
  We use the notation of Proposition~\ref{prop: product structure}.
  \begin{enumerate}
  \item It is clear from Proposition~\ref{prop: product structure},
    that the closed orbit $H\Lambda$ is of finite volume if and only
    if $Z=Z_a$. 
  \item\label{r.ait.2} The group $Z_a$ in Proposition~\ref{prop:
      product structure} is the center of $H_1$ and so is  a  
 $\Q$-group itself (with respect to the $\Q$-structure induced by
 $\Lam$). Moreover, it has no non-trivial $\Q$-characters as these
 will induce corresponding ones on $H_1$ because $H_1$ is an almost direct
 product $H_1=Z_a\cdot S$. By the Borel Harish-Chandra Theorem it
 follows that $Z_a\Lam$ is of finite volume (which in this case means
 compact) or in other words, if we denote
 $Z_\Lambda=\operatorname{Stab}_Z(\Lambda)$ then $Z_\Lambda$ is a
 lattice in $Z_a$. As $Z_a\subset A\cong \R^{n-1}$ we conclude that in
 particular,  the discrete subgroup 
 $Z_\Lambda$ is a finitely generated free abelian group with $\on{rank} (Z_\Lambda)=\dim Z_a$.
 \item Combining (1),(2) we conclude that the orbit $H\Lam$ is of
   finite volume if and only if $\on{rank}(Z_\Lambda)=\dim Z$. 
 \end{enumerate}
\end{remark}

\section{Intermediate lattices}
\name{subsection: types}
We now introduce intermediate 
lattices, and the homogeneous subspaces they belong to, in
detail. This builds on and expands earlier work of several authors,
see \cite{LW, Tomanov,McMullen, ELMV}. Our approach is close to that
of \cite{McMullen}, in that we emphasize the structure of algebras of
matrices associated with a lattice.  
We introduce for any lattice an \textit{associated algebra}. 
In \S\ref{sec: recognizing} we characterize intermediate lattices and
the homogeneous spaces they belong to by simple algebraic properties
of the associated algebra.  
In \S\ref{subsec: constructions} we explain some constructions of
lattices, and show using the aforementioned characterization, that the
constructions give rise to all intermediate   
lattices. In turn, this gives rise to an explicit construction of all 
homogeneous $A$-invariant measures. These results will be an 
important ingredient in the proof of Theorem~\ref{u thm 2}.

\subsection{$\Q$-algebras}\label{sec: algebras}
Let $F_j, \ j=1\dots r$ be number fields 
and consider the direct sum $B=\oplus_{j=1}^r
F_j$. Equipped with coordinate-wise addition and multiplication, $B$ is a finite dimensional
$\Q$-algebra. By the Artin-Wedderburn Theorem, any commutative finite
dimensional  semisimple 
$\Q$-algebra is of the above form. 

By a homomorphism between two such algebras we shall mean a map that
respects the algebraic operations and sends the identity of one
algebra to the identity element of the other. If $B$ is an algebra as
above, then an algebra $B'\subset B$ will be referred to as a
subalgebra if the inclusion $B'\hookrightarrow B$ is a homomorphism;
in particular, $B$ and $B'$ share the same unit. A subalgebra $B'\subset B$
will be referred to as a subfield if it is a field. 
We emphasize that  if $B=\oplus_1^r F_j$ as above, with $r>1$, then
the $F_j$'s are not subalgebras nor subfields.  

The theory of algebras of the above form is almost completely
analogous to the theory of number fields with only minor adaptations
resulting from the fact that we deal with direct sums of number
fields. For example it is clear that if $B=\oplus_1^rF_j$ is an
$n$-dimensional $\Q$-algebra, then it has exactly $n$ distinct  
homomorphisms into $\C$ and those are obtained by first projecting to
the components $F_j$ and then composing with the various embeddings of
the fields $F_j$ into $\C$.

\subsection{The associated algebra}
Let $D$ denote the algebra of $n \times n$ diagonal real matrices.
For $i=1, \ldots, n$, let $p_i: D \to \R$ be the algebra homomorphism
$\diag(d_1, \ldots, d_n) \mapsto d_i.$  
Given a partition 
$\PPP$ as in~\eqref{eq: partition0} we denote by $D(\PPP)$ the subalgebra of $D$ defined by
$$
D(\PPP) \df \left \{x \in D: p_{i}(x)=p_{j}(x)\textrm{ whenever } i\sim_\PPP j \right \}.
$$ 
\begin{dfn}\label{basic def}
Let $\Lambda \subset \R^n$ be a lattice.
\begin{enumerate}
\item We denote $V_{\Lambda} \df \spa_{\Q} (\Lambda)$.
\item 
For any partition $\PPP$ we define the \textit{associated algebra of
  $\Lam$ with respect to 
$\PPP$} to be
$$ \A_{\Lambda}(\PPP) \df \{a \in D(\PPP): a V_{\Lambda} \subset V_{\Lambda} \}.$$ 
We denote by $\PPP_0$ the partition into singletons,  denote
$\A_{\Lambda}(\PPP_0)$ simply by $\A_{\Lambda}$, and refer to it as
the \textit{associated algebra} to $\Lam$. 
\item Given a subalgebra $B\subset \A_\Lambda$ we define the
  \textit{associated partition} $\PPP_B$ to be the partition of
  $\set{1,\dots, n}$ induced by the equivalence relation
$$i\sim j \ \Longleftrightarrow \ p_i|_B \sim p_j|_B.$$ 
A partition of the form $\PPP_B$ will be referred to as an 
\textit{algebra partition for $\Lam$} and in case 
$B\subset \A_{\Lam}$ is a subfield, as a \textit{field partition for $\Lam$}. 
\end{enumerate}
\end{dfn}

 The following result demonstrates the usefulness of the
 associated algebra for the study of the $A$-action on $\Xn$: 
\begin{thm}\label{n.c.c}
Let $\Lam\in\Xn$, $\PPP$ a partition, and $H=H(\PPP)$. 
 The orbit $H\Lam$ is closed if and only if
 $\dim_{\Q}\A_{\Lam}(\PPP)=|\PPP|$. Furthermore, the orbit is of
 finite volume if and only if in addition, the associated algebra
 $\A_\Lambda(\PPP)$ is a field. 
\end{thm}
Theorem \ref{n.c.c} is a compressed version 
of Theorem~\ref{prop: addition}, which is the main result of this
section. Theorem \ref{prop: addition} will be stated and proved below 
after some more preparations. 
 As Theorem \ref{n.c.c} shows, we will need to understand the dimension of
 associated algebras over $\Q$. 
Note that the elements of $\A_\Lambda(\PPP)$ are simply the rational
matrices in $D(\PPP)$ with respect to the rational structure induced
by $\Lam$.  
The associated algebra of a lattice $\Lam$ is a commutative algebra
over $\Q$. It is finite-dimensional because it can be can conjugated
into $\on{Mat}_{n\times 
  n}(\Q)$.

\begin{prop}\label{dimensions are equal}
For any lattice $\Lam$ and any partition $\PPP$ we have that
$\dim_{\Q}\A_\Lambda(\PPP)=\dim_{\R}
\pa{\A_\Lambda(\PPP)\otimes_{\Q}\R}$. In particular,  
$\dim_{\Q}\A_\Lambda(\PPP)\le |\PPP|$ with equality if and only if
$D(\PPP)$ is a subspace of $\on{Mat}_{n\times n}(\R)$ which is defined
over $\Q$ with respect to  
the $\Q$-structure induced by $\Lam$.
\end{prop}
\begin{proof}
It is well known that for any number field $F$ one has the equality
$\dim_{\Q}F=\dim_{\R}\pa{F\otimes_{\Q}\R}$. 
It is evident that $\A_{\Lambda}(\PPP)$ is a semisimple algebra
and so by the Artin-Wedderburn theorem, is isomorphic to a direct sum of number fields. 
Thus the first part of
the Proposition follows from the fact that the associated algebra is
isomorphic to a direct sum of number fields. The dimension bound
follows from the natural 
inclusion $\A_\Lambda(\PPP)\otimes_{\Q}\R\subset D(\PPP)$. Finally, as noted above, with
respect to the $\Q$-structure induced by $\Lam$, $\A_\Lam(\PPP)$
consists of exactly the rational points of $D(\PPP)$ and therefore, by
the first part, $D(\PPP)$ has a basis consisting of rational matrices
if and only if $\dim_{\Q}\A_\Lambda(\PPP)=\dim_{\R} D(\PPP)$. 
\end{proof}

Since matrices in $D(\PPP)$ commute with matrices
in $H(\PPP)$, we have:
\begin{prop}\label{constant alg}
The assignment $\Lam\mapsto \A_\Lambda(\PPP)$ is 
constant along $H(\PPP)$-orbits. Therefore, $\A_\Lambda(\PPP)$ is an
invariant attached to the orbit $H(\PPP)\Lam$. 
\end{prop}
\qed

\begin{thm}\name{prop: algebra}
Let $\Lam\in\Xn$ be a lattice and $B\subset \A_\Lambda$ a
subalgebra. Then there is an isomorphism of $\Q$-algebras $\varphi:
\bigoplus_{j=1}^r F_j \to B,$ where the  $F_j$'s are totally real
number fields of degrees  $d_j\df\deg(F_j/\Q)$ such that  
\eq{eq: partition}{
\PPP_B = \bigsqcup_{\displaystyle_{k=1, \ldots, d_j}^{j=1,\ldots, r}} 
I_{j,k}
} 
where 
\begin{enumerate}
\item\label{alg.1}
For each $j$, the number $s_j \df | I_{j,k}|$ is independent of $k$.  
\item\label{alg.2}
For each $j,k$, there is a field embedding $\sigma: F_j \to \R$ such
that for all $i \in I_{j,k},  \, \sigma = p_i \circ \varphi|_{F_j}.$
Moreover any field embedding of $F_j$ appears for some choice of $k$.  
\end{enumerate}
\end{thm}
Before proving Theorem \ref{prop: algebra} we 
deduce a characterization of the partitions that
Theorem~\ref{n.c.c} may be applied to.  
\begin{cor}\label{algparchar}
Let $\Lam\in\Xn$ be given. A partition $\PPP$ is an algebra partition
for $\Lam$ if and only if $\dim_{\Q}\A_\Lambda(\PPP)=|\PPP|$ and in
that case,
$$
\PPP=\PPP_B, \ \ \mathrm{where} \ B=\A_\Lam(\PPP). 
$$

\end{cor} 
\begin{proof}[Proof of Corollary \ref{algparchar}]
Suppose $\PPP$ is an algebra partition for $\Lam$; that is, there
exists a subalgebra $B\subset \A_\Lambda$ such that $\PPP=\PPP_B$. 
It follows from the definition that $\A_{\Lam}(\PPP_B)\supset B$ so by
Proposition~\ref{dimensions are equal}, in order to conclude that
$\dim_{\Q}\A_\Lambda(\PPP_B)=|\PPP_B|$, it is enough 
to show that $\dim_{\Q}B=|\PPP_B|$. The latter statement follows from the
description of $\PPP_B$ given in Theorem~\ref{prop: algebra}. 

In the other direction, assume that $\PPP$ satisfies
$\dim_{\Q}\A_\Lambda(\PPP)=|\PPP|$ and denote
$B=\A_\Lambda(\PPP)$. Then it follows from the definitions that  
$\PPP$ refines $\PPP_B$. Again, by Theorem~\ref{prop: algebra} we
deduce that $\dim_{\Q}B=|\PPP_B|$ and so the partitions $\PPP,\PPP_B$
are equal. 
\end{proof}

For the proof of Theorem~\ref{prop: algebra} we will require the
following well-known fact (for which we were unable to find a
reference).  

\begin{lem}\name{lem: norm}
Let $F$ be a number field of degree $d$ over $\Q$, let $\sigma_i: F \to \C, \, i=1, \ldots, d$ be its distinct embeddings in $\C$, and let $k_1, \ldots , k_d \in \Z$ such that for all $x \in F$, $\prod_1^d \sigma_i(x)^{k_i} \in \Q$. Then all the $k_i$ are equal. 
\end{lem}
\begin{proof}
We will prove that $k_1 = k_2$, the other cases being symmetric to this one. Assume by contradiction that $k_2 > k_1$. Since the norm map $N(x) = \prod \sigma_i(x)$ has its values in $\Q$, we may divide through by $N(x)$ to assume that $k_1=0 < k_2$. 
Choose a basis $\alpha_1, \ldots, \alpha_d$ of $F$ over $\Q$, and denote by $\varphi$
the rational map 
$$(X_1, \ldots, X_d) \mapsto \prod_{i=1}^d \sigma_i\left(\sum_j \alpha_j X_j \right)^{k_i}, $$
which we can simplify as 
$$
\varphi(\vec{X}) = \prod_{i=1}^d L_i(\vec{X})^{k_i}, \ \ \mathrm{where} \ L_i(\vec{X}) \df \sum_j  \sigma_i(\alpha_j) X_j.
$$
The $L_i$ are linearly independent linear functionals. Thus the zero set of $\varphi$ is the union of the kernels of those $L_i$ for which $k_i \neq 0$; in particular $\varphi$ is identically zero on $\ker ( L_2)$ but not on $\ker (L_1)$. 

Now let $\sigma: \C \to \C$ be a field automorphism such that $\sigma_{1} = \sigma \circ \sigma_{2}$.  Then for each $\vec{X} \in \Q^d$ we have $\varphi(X) \in \Q$, hence $\sigma \circ \varphi (\vec{X}) = \varphi(\vec{X})$, and since $\Q^d$ is Zariski dense in $\C^d$, this implies that $\sigma \circ \varphi$ and $\varphi$ are identical as rational maps. On the other hand $\sigma_i \mapsto \sigma \circ \sigma_i$ is a permutation $\sigma_i \mapsto \sigma_{\pi(i)}$ with $\pi(2)=1$. This means that $\varphi(X)$ can also be written as $\prod_1^d L_{\pi(i)}(X)^{k_i}$, and so 
 $\varphi$ is identically zero on $\ker( L_1)$ --- a contradiction. 
\end{proof}

\begin{proof}[Proof of Theorem \ref{prop: algebra}]
The fact that $B$ is isomorphic to a direct sum of number fields is a
consequence of the Artin-Wedderburn Theorem and follows from the fact
it is a finite dimensional semisimple $\Q$-algebra (see also
\cite[Prop. 3.1]{Tomanov}).  
So we have an abstract isomorphism  $\varphi: \bigoplus_{j=1}^r F_j \to B $, where the 
$F_j$'s are number fields, and for each $i$, consider the restriction
of $p_i$ to $B$, which we continue to denote by $p_i$. Since the
diagonal embedding of $\Q$ as scalar matrices is a subalgebra of $B$,
each $p_i$ is non-zero. For each $j$, let $1_j$ denote the image of $1
\in F_j$ in $B$. Then for $j \neq j'$ we have $1_j \cdot 1_{j'}
=0$. Since $\R$ has no zero-divisors, for each $i$ there is a unique
$j$ such that $p_i(1_j) \neq 0$.  This implies that $p_i \circ
\varphi|_{F_j}:F_j\to\R$ is a  
non-zero map that respects addition and multiplication. It follows that it is  a real field embedding. 

To prove assertions \eqref{alg.1},\eqref{alg.2} it remains
to show that for each $j$, and each field embedding $\sigma: F_j \to
\C$, the number of indices $i$ for which $\sigma = p_i \circ \varphi$
is a nonzero number independent of $\sigma.$ To this end, for each $x
\in F_j$ let  
\eq{def psi}{
\psi(x)\df \varphi(x)+\sum_{j'\ne j}1_{j'}\in B\subset \A_\Lambda.
}
This is a diagonal matrix whose $i$-th entry is 1 if $p_i \circ
\varphi|_{F_j}$ is zero, and is $p_i \circ \varphi(x)$ otherwise. In
particular, $\det \psi(x)$ is the product of the numbers $p_i \circ
\varphi(x)$, taken over the indices $i$ for which $p_i \circ
\varphi|_{F_j}$ is not zero, and is a rational number, since
$V_{\Lambda}$ --- on which  
$\psi(x)$ acts --- is a $\Q$-vector space. So the claim follows from
Lemma \ref{lem: norm}.  

\end{proof}

\subsection{Recognizing intermediate lattices}\name{sec: recognizing}
We are now in a position to prove the main result of this section. 
\begin{thm}\name{prop: addition}

Let $\Lambda  \in \Xn$, let $\PPP$ be a partition, and let $H= H(\PPP), \, Z=Z(\PPP)$ be the corresponding groups. The following are equivalent:
\begin{enumerate}
\item
$H\Lambda$ is closed in $\Xn$. 
\item
$H$ is defined over $\Q$ with respect to the $\Q$-structure induced by $\Lam$. 
\item
$Z$ is defined over $\Q$ with respect to the $\Q$-structure induced by $\Lam$.
\item
$\dim_{\Q} \A_{\Lambda}(\PPP)=|\PPP|$. 
\end{enumerate}
Moreover,  the orbit $H\Lam$ is of finite volume if and only if the
associated algebra $\A_\Lambda(\PPP)$ is a field of degree $|\PPP|$
over $\Q$. 
\end{thm}

\begin{proof}[Proof of  the first part of Theorem~\ref{prop: addition}]
Recall the notation of Proposition~\ref{prop: product structure}:
$Z=Z(\PPP), S=S(\PPP), \, H_\Lambda=\on{Stab}_H(\Lam), \, 
H_0=\on{Zcl}(H_\Lam), \, H_1=H_0^\circ$, where $\on{Zcl}$ stands for Zariski closure. 
\ignore{
\eqref{c.o.1}. The argument giving the proof of part~\eqref{c.o.1} of
the Theorem goes through a third equivalent condition that links
between the two conditions in the  
statement; we show that both conditions are equivalent to the fact
that the center $Z$ of $H$ is defined over $\Q$ with respect to the
$\Q$-structure induced by $\Lam$. 
We explicate about this terminology: A matrix is said to
be$\Lam$-rational if when presented with respect to a basis of $\Lam$
it has rational entries, or in other words, if it preserves the space
$V_\Lambda$. An algebraic group $F\subset G$ is said to be defined
over $\Q$ with respect to the $\Q$-structure induced by $\Lam$ if it
contains a Zariski dense set of $\Lam$-rational matrices.  

As a first step we show that $Z$ is defined over $\Q$ if and only if
$\dim_{\Q}\A_\Lambda(\PPP)=|\PPP|$. Recalling our definitions, we see
that 
$$\set{a\in Z: a\textrm{ is }
  \Lam\textrm{-rational}}=\set{a\in\A_\Lambda(\PPP):\det a=1}.$$

Throughout this proof, we denote $Z' \df g^{-1}Zg$ and $H' \df
g^{-1}Hg$. Note that since left multiplication of $\Xn$ by $g$ is a
homeomorphism, the statement that $H\Lambda$ is closed is equivalent
to the statement that $H'\Gamma$ is closed in $G/\Gamma$.  
}
Throughout the proof we will only refer to the $\Q$-structure induced by $\Lam$.

\noindent(1) $\implies$ (3): 
As we saw in Proposition~\ref{prop: product structure}, $H_0$ is a
$\Q$-algebraic group, contains $H_1$ as a finite-index subgroup,  and
moreover $S\subset H_1=Z_a\cdot S\subset H$ and so   
$H_0$ is reductive.
Let $G_1$ denote the centralizer of $H_1$ in $G$, which is again a
reductive $\Q$-algebraic group containing $Z$. By~\cite[Proposition
10.15]{Raghunathan} this implies that the orbit $G_1\Lam$ is closed
and hence can be viewed as the quotient of a reductive $\Q$-group by
its integral points. We  
now claim that $Z\subset G_1$ is a maximal $\R$-diagonalizable group
and that $Z\Lam$ is a closed orbit. 
Assuming this, applying \cite[Theorem 1.1]{with george},  we find that $Z$ is
a $\Q$-subgroup of $G_1$, and hence a $\Q$-subgroup of $G$.  Thus the
claim implies (3). 

Because $H,H_0$ differ only in $Z_s$ which is in the 
center $Z$, we have that  $Z=H \cap G_1=A\cap G_1$. As $G_1$ is
normalized by $A$ we deduce that $Z$ is a maximal $\R$-diagonalizable
subgroup of $G_1$.  
Since both orbits $H \Lam$, $G_1 \Lam$ are closed, by \cite[Lemma
2.2]{Shah} so is $(H\cap G_1)\Lam=Z\Lam$.

(3) $\implies$ (2): As  $H$ is the connected component of the identity
in the centralizer of $Z$ in $G$, if $Z$ is a $\Q$-algebraic group, so
is $H$.

(2) $\implies$ (1):  Since $H$ is reductive and defined over $\Q$, it
follows from \cite[Proposition 10.15]{Raghunathan} that $H\Lam $ is
closed in $\Xn$. 

(4) $\implies$ (3): It follows from Proposition~\ref{dimensions are
  equal} that if $\dim_{\Q}\A_\Lambda(\PPP)=|\PPP|$ then $D(\PPP)$ is
defined over  
$\Q$. As $Z=D(\PPP)\cap G$ we conclude that $Z$ is defined over $\Q$.

(3) $\implies$ (4): 
 If $Z$ is defined over $\Q$,  it contains a Zariski dense subset of
 $\Lam$-rational matrices.  These are by definition elements of $\A_\Lambda(\PPP)$ so
 we conclude that the dimension of the real vector space
 $\on{Zcl}(\A_\Lambda(\PPP))$ is at least the dimension of $Z$, which
 is  
$|\PPP|-1$. On the other hand, the line of scalar matrices is always
in this space (regardless of $\Lam$) and so the dimension is $|\PPP|$
as desired. 
\end{proof}
In order to complete the proof of Theorem~\ref{prop: addition} we will
need the following  Proposition which relates the structure of
$\A_{\Lambda}$ with the structure of the stabilizer of $\Lambda$ in
$Z(\PPP)$.   
\begin{prop}\name{prop: more information}
Let $\Lam\in\Xn$, $\PPP$ a partition,  and $Z=Z(\PPP)$.  
Let $F_1, \ldots, F_r$ and $\varphi$ be as in Theorem \ref{prop:
  algebra} applied to  $B=\A_\Lam(\PPP)$. Then the group 
$$Z_{\Lambda} \df \{a \in Z: a \Lambda = \Lambda\}$$
 is contained as a finite index subgroup in $\varphi \left(\prod_1^r
   \mathcal{O}^{\times}_j \right)$, where $\mathcal{O}^{\times}_j$ is
 the (multiplicative) group of units of the ring of integers of $F_j$.  
\end{prop}
\begin{proof}
We first prove the inclusion $Z_{\Lambda} \subset \varphi\left( \prod_1^r \mathcal{O}^{\times}_j \right).$ 
Suppose $a \in Z_{\Lambda} \subset \A_{\Lambda}$. In light of Theorem
\ref{prop: algebra} there are $x_j \in F_j$ such that for each $i \in
I_{j,k}$, $p_i(a) = \sigma_k(x_j)$, where the $\sigma_k$ are the
distinct field embeddings of $F_j$. We need to show that each $x_j$ is
a unit in the ring of integers of $F_j$. Let $M$ be a matrix
representing the action of $a$, with respect to a basis of $\R^n$
which generates $\Lambda$. Since $a$ preserves $\Lambda$, $M$ has
integral entries, and has $x_j$ as an eigenvalue. Thus $x_j$ is a root
of the characteristic polynomial of $M$ which is a monic polynomial
over the integers, and so is an algebraic integer. 
As the same argument applies to $a^{-1},x_j^{-1},M^{-1}$ we conclude that $x_j$ is a unit.

We now show that $Z_{\Lambda}$ is of finite index in this inclusion. For each $j$, let $d = d_j \df \deg(F_j/\Q)$.
By Dirichlet's theorem, $\mathcal{O}^{\times}_j$ contains $d-1$
multiplicatively independent elements $\alpha_1, \ldots,
\alpha_{d-1}$, so it suffices to show that a finite power of each
$M_i$ fixes $\Lambda$, where
$M_i=\psi(\alpha_i)\in\A_{\Lambda}(\PPP)$, where $\psi$ is as
in~\equ{def psi} above.   

To this end, write $\alpha = \alpha_i, \, M = M_i$, and note that
by~\equ{def psi} and Theorem \ref{prop: algebra} the characteristic
polynomial $p_M(X)$ of $M$ is 
of the form $$p_M(X)=[ m_\alpha(X)]^{b_1}[X-1]^{b_2},$$ where
$m_\alpha(X)$ denotes the minimal polynomial of $\alpha$ and $b_1,b_2$
are non-negative 
integers. In particular,  $p_M(X)$ has coefficients in $\Z$ and degree
$n$. This implies that the additive group $\til \Lambda$ generated by  
$\bigcup_{k=0}^{n-1} M^k \Lambda$
is $M$-invariant. Representing $M$ with respect to a basis of
$V_{\Lambda}$ contained in $\Lambda$, we see that $M$ has rational
coefficients, and in particular $\til \Lambda$ is discrete, i.e.\ is a
lattice in $\R^n$. Since $\til \Lambda$ contains $\Lambda$, it must
contain it as a subgroup of finite index, and since $\det M = \pm 1$,
the index is preserved by the action of $M$. Since $\til \Lambda$
contains only finitely many subgroups of a given index, there is a
power of $M$ preserving $\Lambda$, as required.  
\end{proof} 
The following Corollary verifies the last statement of Theorem~\ref{prop: addition}.
\begin{cor}\name{r.21.6}
Let $H\Lam$ be a closed orbit of the block group $H=H(\PPP)$ and let
$Z=Z_s\cdot Z_a$ be the decomposition of the center $Z=Z(\PPP)$ given
in Proposition~\ref{prop: product structure}. Then, the number of
summands in the decomposition of the associated algebra as a direct
sum of number fields  
$\A_\Lambda(\PPP)\cong \oplus_1^r F_j$ satisfies $\dim Z_s=r-1$.

 In
particular, the following are equivalent:
\begin{itemize}
\item[(i)]
 $H\Lam$ is of finite volume.

\item[(ii)]
$\A_\Lambda(\PPP)$ is a field. 
\item[(iii)]
$Z\Lam$ is compact. 
\end{itemize}
\end{cor}
\begin{proof}
Denote $\deg(F_j/\Q)=d_j$. By the part of Theorem~\ref{prop: addition}
already established, $\dim Z=\dim Z_s+\dim Z_a=|\PPP|-1=\sum_{j=1}^r
d_j -1.$ 
By  combining Proposition~\ref{prop: more information} and
part~\eqref{r.ait.2} of Remark~\ref{anisotropic torus} we conclude
that 
 $\dim Z_a=\sum_{j=1}^r (d_j-1)=\sum_{j=1}^r d_j - r$.
Combining these two equalities we conclude that $r= \on{dim}(Z_s)-1$ as desired. 

Finally, by Proposition~\ref{prop: product structure} we know that the
closed orbit $H\Lam$ is of finite volume if and only if $\dim Z_s=0$,
which, by the above reasoning, implies the equivalence of (i) and (ii)
and shows that they imply (iii). For the reverse implication (iii)
$\implies $ (i), note that if $Z\Lam$ is compact then $Z_{\Lam}$ is
cocompact in $Z$ which implies that $Z$ is defined over $\Q$ (with
respect to the $\Q$-structure induced by $\Lam$), so we may use the
implication (3) $\implies$ (1) of 
Theorem \ref{prop: addition}, and Proposition \ref{prop: product
  structure}.  
\end{proof}

\section{Consequences and examples} \name{sec: consequences}
Proposition \ref{prop: why block groups} and Theorem \ref{prop:
  addition} furnish a link between the algebraic
properties of intermediate lattices and the structure of their orbits
under block groups. This sheds light on all possible $A$-homogeneous
spaces. In this section we collect results in this direction, and
conclude with some examples.  

\begin{cor}\name{cor: bijective}
For any lattice $\Lambda \in \Xn$, the map $B \mapsto H(\PPP_B)$
is a bijective correspondence between the subalgebras of $\A_{\Lambda}$, and the block groups $H$ for which $H\Lambda$ is a closed orbit. Under this bijection subfields
of $\A_\Lambda$ correspond to finite volume orbits. 
The bijection is order-reversing for the orderings of the corresponding sets by inclusion. 
\end{cor}
\begin{proof}
This follows from Theorem~\ref{n.c.c} and Corollary~\ref{algparchar}.
\end{proof}

In Corollary \ref{cor: bijective}, the trivial algebra $\Q$
corresponds to the block group $H=G$ (the group with one block).  
Recalling Definition~\ref{i.l.d.1}, we obtain: 
\begin{cor}\label{inter char}
A lattice $\Lambda\in\Xn$ is intermediate (resp.\ intermediate of
finite volume type) if and only if the associated algebra $\A_\Lambda$
is nontrivial (resp.\ contains a  subfield other than $\Q$). 
\end{cor}
\qed

In the other extreme $A=H(\PPP_0)$ we have the following Corollary
which explains the terminology in Definition~\ref{i.l.d.2}.   
\begin{cor}
\name{cor: compact A orbits}
A lattice $\Lambda\in \Xn$ is an algebra lattice (resp.\ a number
field lattice) if and only if the associated algebra is
$n$-dimensional over $\Q$ (resp. is a field of degree $n$ over $\Q$).  
\end{cor}
The following Corollary recovers a result from~\cite{LW}. Note that by
Theorem~\ref{prop: algebra} any field partition must be an equiblock
partition; that is, a partition into blocks of equal size. 
\begin{cor}[See \S6 in \cite{LW}]\label{equi must}
 If $H = H(\PPP)$ has a finite volume orbit in $\Xn$ then $H$ is necessarily an equiblock group. 
 \end{cor}
\qed
 


We now summarize the relation between subalgebras and
partitions. 
 Given a lattice $\Lam$ we have defined two maps
\begin{align}\label{correspondences}
&\xymatrixcolsep{5pc}\xymatrix{
\set{\textrm{partitions $\PPP$}} \ar@/^2pc/[r]^{\PPP\mapsto
  \A_\Lambda(\PPP)}&\set{\textrm{subalgebras of
    $\A_\Lambda$}}\ar@/^2pc/[l]_{\PPP_B\mapsfrom B} 
}
\end{align}
The image of the RHS in the LHS of~\eqref{correspondences} is the
collection of algebra partitions which are those of dynamical
interest, in light of Corollary~\ref{cor: bijective}. 
\begin{prop} \name{prop: correspondences}
Let $\Lam$ be a lattice.
\begin{enumerate}
\item Both maps in~\eqref{correspondences} respect the partial
  orderings of refinement on the LHS and inclusion on the RHS.  
\item For any subalgebra $B\subset \A_\Lambda$ we have that
  $B=\A_\Lambda(\PPP_B)$; that is, going from the RHS to the LHS and
  back in~\eqref{correspondences} is the identity map. 
\item In the other direction, for any partition $\PPP$, if
  $B=\A_\Lambda(\PPP)$ then $\PPP_B$ is the finest algebra partition
  which is coarser than $\PPP$. 
\end{enumerate}
\end{prop} 
We will not be using Proposition  \ref{prop: correspondences} and its proof is
left to the reader.
\qed

\subsection{Density properties}
It is a well-known result of Prasad and Raghunathan
\cite{PrasadRaghunathan}, based on earlier work of Mostow, that the
set of compact $A$-orbits is dense in any fixed finite volume orbit $H\Lam$ of a block group $H$. Our results imply the
following related result:
\begin{prop}\name{prop: density}
Let $H_1\varsubsetneq H_2$ be two equiblock subgroups of $G$ and let
$\Lambda_0$ be an intermediate lattice such that both orbits
$H_i\Lambda_0$ are homogeneous and of finite volume.  
Let $\PPP_i$ denote the partition satisfying $H_i=H(\PPP_i)$ and let
$K=\A_{\Lambda_0}(\PPP_1)$ be the subfield of the associated algebra
to $\Lambda_0$  
corresponding to $\PPP_1$ by Corollary~\ref{cor: bijective}.
Then the set 
$$
\left\{\Lambda \in H_2\Lambda_0 : H_1\Lambda \mathrm{\ is \ a \
    homogeneous \ subspace \ and \ } K \cong \A_{\Lambda}(\PPP_K)  
\right\}
$$
is dense in $H_2\Lambda_0$.

\end{prop}

\begin{proof}
 Let $Z$ be the center of $H_1$. By Corollary \ref{r.21.6},
 $Z\Lambda_0$ is compact. Let $\Gamma'$ be the subgroup of $H_2$
 fixing $\Lambda_0$, which is an arithmetic lattice in $H_2$. Let
 ${H_2}(\Q)$ denote the elements $q \in H_2$ which are rational with
 respect to the corresponding $\Q$-structure, a dense subgroup of
 $H_2$ (see \cite[Prop. 3.4]{LW} for more details). For each $q \in
 H_2(\Q)$, $q\Gamma' q^{-1}$ is commensurable with $\Gamma'$, so the
 orbits $Z\Lambda_0$ and $Zq\Lambda_0$ share a common finite
 cover. This implies that each $Zq\Lambda_0$ is also compact. Another
 application of Corollary \ref{r.21.6} shows that $H_1q\Lambda_0$ is also a homogeneous
 subset of finite volume.  Moreover, for each $q \in H_2(\Q)$, $V_{\Lambda_0} =
 V_{q\Lambda_0}$. This implies that $\A_{\Lambda_0}(\PPP_K) =
 \A_{q\Lambda_0}(\PPP_K)$. So the set of lattices $\{q\Lambda_0: q \in  
H_2(\Q)
 \}$ has the desired properties. 
\end{proof}

\subsection{Constructing 
intermediate lattices}\name{subsec: constructions}

Let $B=\oplus_1^r F_j$ be an $n$-dimensional $\Q$-algebra where the
$F_j$'s are totally real number fields of degrees $d_j$ over $\Q$
respectively; so $\dim_{\Q} B=\sum_1^r d_j= n$.  
Let $\sigma_i: B \to \R, i=1\dots n$ be some enumeration of the $n$
distinct homomorphisms of $B$ into the reals. More concretely, if we
denote by $\tau_{jk}:F_j\to\R, \ k=1,\dots, d_j$ the various field
embeddings of  
$F_j$ into the reals, and view each $\tau_{jk}$ as a homomorphism from
$B$ to $\R$, then $\sigma_1,\dots,\sigma_n$ is some enumeration of the
$\tau_{jk}, \, j=1,\dots, r, k=1,\dots, d_j$. Let $v: B\to \R^n$ be the map   
\begin{equation}\label{g.emb}
\alpha\mapsto v(\alpha) \df (\sigma_1(\alpha), \ldots, \sigma_n(\alpha)) \in \R^n.
\end{equation}
 Let $L$ be an  additive subgroup of $B$ of rank $n$. As
 $B\otimes_{\Q}\R=\R^n$,  the group $\{v(\alpha): \alpha \in L\}$ is a
 lattice in $\R^n$. Let  
\eq{eq: vectors first type}{
  \Lambda_L \df c_L \, \{v(\alpha): \alpha \in L\},}
  where $c_L$ is chosen so that $\Lam_L$ has covolume 1 and so belongs
  to $\Xn$. 
 Lattices arising in this way have been studied by many authors
 (mainly in the case where $B$ is a field), see
 e.g. \cite[Chap. 1]{GL} or \cite[p. 54]{PR}. We refer to below  to
 lattices of the form $\Lam_L$ as \textit{lattices arising via}
 \eqref{eq: vectors first type}.  

The following proposition gives an explicit construction of all algebra lattices.
 \begin{prop}\label{prop: compact A orbits}
 \begin{enumerate}
\item Let $\Lam_L$ be a lattice arising via~\eqref{eq: vectors first
    type} with $L$ a rank-$n$ subgroup of the $n$-dimensional
  $\Q$-algebra  $B$ as above. Then the associated algebra
  $\A_{\Lam_L}$ is isomorphic to $B$. In particular,  the orbit
  $A\Lam_L$ is closed and so consists of algebra lattices.  
\item If $\Lam$ is an algebra lattice, then there is  a lattice
  $\Lam_L$ arising via~\eqref{eq: vectors first type} with 
  $\Lam\in A\Lam_L$.
\end{enumerate}
 \end{prop}
 \begin{proof}
 (1)  The map $v:B\to\R^n$ in~\eqref{g.emb} is $\Q$-linear and also
 respects multiplication in the sense that $v(\alpha\cdot\beta) =
 \on{diag}\pa{\sigma_1(\alpha),\dots,\sigma_n(\alpha)}\cdot
 v(\beta)$. This means that $V_{\Lam_L}=v(B)$, and moreover that
 the map $\alpha\mapsto
 \on{diag}\pa{\sigma_1(\alpha),\dots,\sigma_n(\alpha)}$ is an
 embedding of $B$ into $\A_{\Lam_L}$. As $B$ is $n$-dimensional and
 $\A_{\Lam_L}$ is at most $n$-dimensional, the above
 map is an isomorphism between $B$ and $\A_{\Lam_L}$. By
 Corollary~\ref{cor: compact A orbits}, $\Lam_L$ is an algebra lattice
 as desired.  
 
 \noindent(2)
 Let $B=\A_\Lambda$. By Corollary~\ref{cor: compact A orbits}, $B$ is
 $n$-dimensional. Choose a vector $w\in \Lam$ all of whose coordinates
 are positive and consider the map 
 $\psi: B\to V_\Lambda$ given by $a\mapsto a\cdot w$. This map is
 clearly $\Q$-linear and injective. As $V_\Lambda$ is $n$-dimensional
 as well, it must be an isomorphism of  
 $\Q$-vector spaces. Let $L\df\psi^{-1}(\Lam)\subset B$. By
 Theorem~\ref{prop: algebra}, by projecting $B$ to the diagonal
 coordinates we obtain an ordering of all the various homomorphisms of
 $B$ into $\R$. This way we obtain a map $v:B\to\R^n$ as
 in~\eqref{g.emb}, and the lattice $\Lam_L$ which
 arises  
 via~\eqref{eq: vectors first type} satisfies $\Lam_L=a\cdot
 \Lam$. Here $a$ is the diagonal matrix obtained by rescaling the diagonal
 matrix $\on{diag}(w_i)$ to have determinant 1. Indeed $a\in
 A$ as we chose $w$ so that all of its coordinates are positive. 
 \end{proof}
 The following Proposition gives an explicit construction of
 all homogeneous $A$-invariant spaces in $\Xn$ (or equivalently, of
 all intermediate lattices).  
 It shows that each such homogeneous space $H(\PPP)\Lam$ contains an
 algebra lattice $\Lam_L$ arising via~\eqref{eq: vectors first
   type}. Moreover, by Corollary~\ref{cor: bijective}, the partition 
 $\PPP$ must be the algebra partition that corresponds to the
 subalgebra of $\A_{\Lam_L}$ that we associate to the homogeneous
 space.  
 \begin{prop}\name{pilc}
 Let $H\Lam$ be a closed orbit for the group $H=H(\PPP)$. Then there
 exists a lattice $\Lam_L$ arising via~\eqref{eq: vectors first type}
 such that $\Lam\in H\Lam_L$. If $H\Lam$ is of finite volume then
 $\Lam_L$ can be taken to be a number field lattice. 
 \end{prop}
 \begin{proof}
 By Proposition~\ref{prop: product structure} (and its notation),
 we may present $H$ as an almost direct product $H=Z_s\cdot
 Z_a\cdot S$. Since $S$ is a semisimple group defined over $\Q$, with
 respect to the $\Q$-structure induced by $\Lambda$,
 Proposition~\ref{thm: BHC} implies that the orbit $S\Lam$ is of
 finite volume. Let $A_0 \df S \cap A$. 
By the theorem of Prasad and Raghunathan \cite{PrasadRaghunathan} the
finite volume orbit $S\Lambda$ contains a 
lattice $\Lambda'$ with a compact $A_0$-orbit. By
Remark~\ref{anisotropic torus}\eqref{r.ait.2} we have that
$Z_a\Lam'$ is compact 
as well and since $Z_a$ commutes with $A_0$, the orbit
$Z_a A_0 \Lam'$ is compact. Applying part (3) of
Proposition~\ref{prop: product structure} we deduce that the orbit
$Z_sZ_a\, A_0\Lam'$ is closed but as $A=Z_sZ_a\, A_0$ we
conclude that $\Lam'$ is an algebra lattice. Similarly, when $H\Lam$
is of finite volume then $A\Lam'$ is compact so $\Lam'$
is a number field lattice. By Proposition~\ref{prop: 
  compact A orbits} we conclude that we may assume without loss of
generality that $\Lam'=\Lam_L$ is a lattice arising via~\eqref{eq:
  vectors first type}. 
 \end{proof}

\subsection{Indecomposable lattices}
\begin{prop}\name{prop: never decomposable}
Let $\mu$ be a finite $A$-invariant homogeneous measure. Then
$\mu$-a.e. $\Lambda$ is indecomposable. In particular $\kappa_\mu \in
\widehat{\mathbf{MG}}_n$ (where $\kappa_\mu$ is as in \equ{eq: kappa}). 
\end{prop}
 
\begin{proof}
A decomposable lattice $\Lambda = \Lambda_1 \oplus \Lambda_2$ has
nonzero vectors with zero entries (namely those nonzero vectors in
each $\Lambda_i$, as embedded in $\Lambda$). If the set of
decomposable lattices had positive $\mu$-measure, then for some index
$i_0$, there would be
a set of positive measure of lattices $\Lambda$ containing a vector
whose $i_0$-th coordinate vanishes. Assume to simplify notation that
$i_0=1$, then by Proposition \ref{prop: Mahler}, such
lattices have a divergent trajectory under the one parameter subgroup $a_t =
\diag(e^{(n-1)t}, e^{-t}, \ldots, e^{-t})$ as $t \to \infty$. This
contradicts the Poincar\'e
recurrence theorem, which asserts that with respect to an invariant
probability measure for an $\R$-action, almost
every point $x$ returns to any neighborhood of $x$ along an unbounded infinite
subsequence. 
\end{proof} 

\subsection{Examples}\label{examples} 
\begin{example}\label{ex zn}
Let $\Lam=\Z^n$, that is, $\Lam$ arises via~\eqref{eq: vectors first type}
from the $n$-dimensional $\Q$-algebra $B=\Q^n$ and so $\A_\Lam\cong
\Q^n$. Moreover,  
for any partition $\PPP$, the subalgebra $\A_\Lambda(\PPP)$ consists
of all diagonal matrices with rational diagonal entries that are
constant in each block of $\PPP$ and so is a $|\PPP|$-dimensional
subalgebra of $\A_\Lambda$. By Corollary~\ref{algparchar} any
partition $\PPP$ is an algebra partition; hence by Corollary~\ref{cor:
  bijective}, the orbit $H(\PPP)\Lam$ is closed for any block group
$H(\PPP)$. Moreover, as $\Q^n$ 
does not have any subfields other than $\Q$, all orbits $H(\PPP)\Lam$
are of infinite volume, apart from the orbit $\Xn$ which is obtained
by choosing the trivial partition that contains only one block, corresponding to the subfield $\Q$.  
\end{example}
\begin{example}\label{ex qsqrt2}
Let $B=F_1\oplus F_2$ be a 4-dimensional $\Q$-algebra where
$F_1=F_2=\Q(\sqrt{2})$. Denote by $x\mapsto x'$ the nontrivial
automorphism of $\Q(\sqrt{2})$.  
Let $L= \mathcal{O}_{F_1}\oplus\mathcal{O}_{F_2}$ and define $\Lam=\Lam_L$ 
to be the lattice defined by~\eqref{eq: vectors first type} where $v(x,y)=(x,x',y,y')$. 
Then
$$\Lam_L=c_L \set{(x,x',y,y'):x,y\in\mathcal{O}_{\Q(\sqrt{2})}}$$ 
is
an algebra lattice with $ F_1\oplus F_2\cong\A_\Lambda$. The
isomorphism is 
given by the 
map $(x,y)\mapsto\on{diag}(x,x',y,y')$. It is not hard to write down a
table of all subalgebras of $\A_\Lambda$ and work out the corresponding
algebra partitions. This gives us a classification of all closed orbits of block
groups through $\Lam$. For example 
if we take $B_1=\set{\on{diag}(x,x,y,y):x,y\in \Q}\cong \Q\oplus\Q$ we
obtain the algebra partition $\PPP_1=\set{\set{1,2},\set{3,4}}$ for
which  
(by Corollary~\ref{cor: bijective}) the orbit $H(\PPP_1)\Lam$ is
closed but of infinite volume because $\Q\oplus\Q$ is not a field.  
On the other hand, if we take $B_2=\set{\on{diag}(x,x',x,x'):x\in
  \Q(\sqrt{2})}$, then we obtain the algebra partition
$\PPP_2=\set{\set{1,3},\set{2,4}}$ for which  
$H(\PPP_2)\Lam$ is a closed orbit of finite volume as $B_2\cong\Q(\sqrt{2})$ is a field.  
\end{example}

\section{Strict inequalities among the $\kappa_{\mu}$}\name{sec: the
  main theorem}
We begin with a definition. 
Let $\varphi:\oplus_{j=1}^{r_2} F^{(2)}_j\hookrightarrow
\oplus_{j=1}^{r_1} F_j^{(1)}$ be an embedding of $\Q$-algebras. 
We say that $\varphi$ is \textit{essential} if 
the image of $\varphi$ projects onto $F_j^{(1)}$ for any $1\le j\le
r_1$. Otherwise we refer to $\varphi$ as \textit{non-essential}.  

We can now state the main result of this section, which is one of the
main results of this paper. 
\begin{thm}\name{main theorem}
Let $\mu_i$, $i=1,2$ be two homogeneous $A$-invariant measures such
that $\on{supp}(\mu_1)\varsubsetneq \on{supp}(\mu_2)$. Let
$H_i=H(\PPP_i)$ and $\Lam\in\Xn$ be such that
$H_i\Lam=\on{supp}(\mu_i)$, so $H_1\varsubsetneq H_2$.
If the containment
$\A_\Lambda(\PPP_2)\subset \A_\Lambda(\PPP_1)$ 
is non-essential then 
$\kappa_{\mu_1} <\kappa_{\mu_2}$.
\end{thm}
\begin{proof}[Deduction of Theorem \ref{u thm 2}] 
 If $\mu_1$ is a finite measure then by Corollary~\ref{cor: bijective}
 the associated algebra to the orbit that supports $\mu_1$ is a field. It
 follows that the   
associated algebra  
to the orbit that supports $\mu_2$ must be a field as well (because it
is a subalgebra of a field) and therefore $\mu_2$ must be a finite
measure as well by another application of Corollary~\ref{cor:
  bijective}. Moreover, as the containment  between the orbits is
strict, it must be that the containment of the associated algebras 
is strict and so the containment is non-essential and
Theorem~\ref{main theorem} applies. We therefore conclude that if
$\mu_1$ is a finite measure, then $\kappa_{\mu_1}<\kappa_{\mu_2}$ and
Theorem~\ref{u thm 2} follows. 
\end{proof}

\begin{example}\label{ex zn2}
Continuing with Example~\ref{ex zn}, note that when we consider the
inclusion of closed orbits $A\Z^n\subset G\Z^n$, the containment of
the associated algebras $\Q\subset \oplus_1^n\Q$ is essential and the
conclusion of Theorem~\ref{main theorem} fails to hold as both of the
generic constants attached to these orbits are equal to 1. 
\end{example}
\begin{example}
Let $\Lam$ be the lattice  constructed in Example~\ref{ex qsqrt2}. In
the notation of that example we know that the orbits $A\Lam,
H(\PPP_1)\Lam$, $H(\PPP_2)\Lam$, $G\Lam$ are all closed  and their
associated algebras are isomorphic respectively to
$\Q(\sqrt{2})\oplus\Q(\sqrt{2})$, $\Q\oplus \Q$, $\Q(\sqrt{2})$, and
$\Q$. We denote the generic values attached to these closed orbits by
$\kappa_0,\kappa_1,\kappa_2,\kappa_3$ respectively, so that
$\kappa_3=\kappa_{\mu_{\Xn}}=1$. Because the inclusions
$\Q\hookrightarrow \Q(\sqrt{2})$,
$\Q\oplus\Q\hookrightarrow\Q(\sqrt{2})\oplus\Q(\sqrt{2})$ are
non-essential  we deduce by Theorem~\ref{main theorem} that
$\kappa_2<\kappa_3$, $\kappa_0<\kappa_1$. On the other hand, the
inclusions $\Q\hookrightarrow \Q\oplus\Q$,
$\Q(\sqrt{2})\hookrightarrow \Q(\sqrt{2})\oplus \Q(\sqrt{2})$ are
essential and so Theorem~\ref{main theorem} does not tell us that the
inequalities $\kappa_1\le \kappa_3$, $\kappa_0\le \kappa_2$ are
strict. Indeed, it is not hard to see that $\kappa_1=1$ because the
lattice $\Z^4$ belongs to the orbit $H(\PPP_1)\Lam$ (as $\Lam$ is the
direct sum of two 2-dimensional lattices and the 2 by 2 blocks of
$H(\PPP_1)$ act on each of the summands separately). We do not know
whether $\kappa_0 < \kappa_2$. 
\end{example}

The proof of Theorem \ref{main theorem} requires some preparations. 
Once again let $\varphi:\oplus_{j=1}^{r_2} F^{(2)}_j\hookrightarrow
\oplus_{j=1}^{r_1} F_j^{(1)}$ be an embedding of $\Q$-algebras. 
We say that $\varphi$ is \textit{aligned}
if $r_1=r_2$.

\begin{prop}\label{aligned rk}
Let $H\Lam$ be a closed orbit of $H=H(\PPP)$ and let $\Lam_L\in H\Lam$
be the algebra lattice constructed in Proposition~\ref{pilc}. Then the
inclusion 
$\A_{\Lam_L}(\PPP)\subset \A_{\Lam_L}$ is aligned.
\end{prop}
\begin{proof}
As both orbits $H\Lam_L$ and $A\Lam_L$ are closed, we may apply
Proposition~\ref{prop: product structure} to both of them and obtain
decompositions of $Z(\PPP)$ and $A$ respectively. Following the proof
of Proposition~\ref{pilc} we see that the split part
in these decompositions may be chosen to be the same; indeed, in the
notation of Proposition~\ref{pilc}, if $Z=Z_s\cdot Z_a$ is the
decomposition for $Z=Z(\PPP)$, then we saw that
$A=Z_s\cdot\pa{Z_a\cdot(A\cap S)}$ and that $\Lam_L$ was chosen so
that $Z_a(A\cap S)\Lam_L$ is compact. 
It now follows from Corollary~\ref{r.21.6} that if we present the
associated algebras to the orbits $H\Lam_L$, $A\Lam_L$ 
as $\A_{\Lam_L}=\oplus_1^{r_1}
F_j^{(1)}$, $\A_{\Lam_L}(\PPP)\cong \oplus_1^{r_2}F_j^{(2)}$, then
$r_1=r_2$.
\end{proof}
\subsection{The Kernel Lemma} 
We will need some more notation. 
\begin{dfn}\label{tilde par}
\begin{enumerate}
\item Given a block group $H=H(\PPP)$ and a closed orbit $H\Lam$ with
  associated algebra $\A_\Lambda(\PPP)\cong \oplus_{j=1}^r F_j$ such
  that 
  $\deg(F_j/\Q)=d_j$,  
we present $\PPP$ as in~\eqref{eq: partition},
$\PPP=\sqcup_{j=1}^r\sqcup_{k=1}^{d_j}  
I_{j,k}.$
Let 
$$
\wt{I}_j\df \cup_{k=1}^{d_j} I_{jk} \ \ \mathrm{and} \ \wt{\PPP}=\sqcup_{j=1}^r 
\wt{I}_j;
$$
that is, the $j$-th block of $\wt{\PPP}$ is obtained by grouping the
diagonal coordinates that correspond to embeddings of $F_j$.  
\item  Given a subset $Q\subset \set{1,\dots, n}$, we 
denote by $\pi_Q:\R^n\to \R^{|Q|}$ the projection to the coordinates of the subset $Q$.
\end{enumerate}
\end{dfn}

\begin{lem}[Kernel Lemma]\name{rl}
Let $H\Lam\subset \Xn$ be a closed orbit of the block group
$H=H(\PPP)$. Let $Q_1,Q_2$ be two blocks of $\PPP$ that are
contained in the same block of  
the partition $\wt{\PPP}$ from Definition~\ref{tilde par}.
Then, there is an automorphism $\rho$ of $\C$ such that for any collection of vectors $v_1,\dots, v_t\in \Lam$, we have 
the following connection between kernels of the  $n\times t$ matrices whose columns are the $\pi_{Q_j}(v_i)$'s
\begin{equation}\label{ranks are equal}
\on{ker}\mat{|&\;&|\\ \pi_{Q_1}(v_1)&\cdots&\pi_{Q_1}(v_t)\\ |&\;&|}=\rho\pa{\on{ker}\mat{|&\;&|\\ \pi_{Q_2}(v_1)&\cdots&\pi_{Q_2}(v_t)\\ |&\;&|}},
\end{equation}
where we let $\rho$ act on $\R^n$ coordinate-wise.
\end{lem}
\begin{proof}
Because of the
block structure of $H$, the conclusion of the Lemma is independent of the
lattice we choose to consider within the orbit $H\Lam$.
By Proposition~\ref{pilc} it is enough to
assume that $\Lam=\Lam_L$ is constructed via~\eqref{eq: vectors first
  type} for some $n$-dimensional $\Q$-algebra $B$ and $L\subset B$.  
By Proposition~\ref{prop: compact A orbits} we may assume that
$B=\A_\Lambda$. Moreover, by the proof of Proposition~\ref{prop:
  compact A orbits}, the map that sends a diagonal matrix in
$\A_\Lambda$ to the vector in $\R^n$ whose coordinates are the
diagonal entries of the matrix is a linear isomorphism\footnote{The
  attentive reader will notice that $V_\Lam$ should be replaced with its dilation.}
between $\A_\Lambda$ and
$V_\Lam$. This means that the statement of the Lemma translates to a
statement about the associated algebra $\A_\Lam$.  

By Proposition~\ref{aligned rk}, the containment
$\A_\Lambda(\PPP)\subset \A_\Lambda$ is aligned.  
We present $\A_\Lambda \cong\oplus_{j=1}^r F_j$,
$\A_\Lambda(\PPP)\cong\oplus_{j=1}^r K_j$ and note that the alignment
of the inclusion of the algebras means that we may assume that for
each $1\le j\le r$ 
the field $F_j$ is an extension of $K_j$ and the inclusion
$\A_\Lambda(\PPP)\subset \A_\Lambda$ is induced from the natural
inclusion $\oplus_1^r K_j\subset \oplus_1^r F_j$.

By Theorem~\ref{prop: algebra}, as $\A_\Lambda$ is $n$-dimensional,
the diagonal coordinates (or the coordinates of $\R^n$) are in one to
one correspondence with the various field embeddings of the fields $F_j$. 
By the above discussion, the blocks $\wt{I}_j$ of $\wt{\PPP}$ are obtained by
grouping the coordinates that correspond to each field $F_j$
together. As $\PPP$ is the algebra partition attached to
$\A_\Lambda(\PPP)$  (see Definition~\ref{basic def} and
Corollary~\ref{algparchar}), the blocks of $\PPP$ are then obtained by
further splitting each $\wt{I}_j$ according to the restriction of the
corresponding embedding of $F_j$ to the subfield $K_j$. That is, two
coordinates that correspond to embeddings of $F_j$ that restrict to
the same embedding of $K_j$ belong to the same block. As the group of
automorphisms of $\C$ acts transitively on the equivalence classes of
embeddings of $F_j$ with respect to the above equivalence relation, we
deduce that if $Q_1,Q_2$ 
are two blocks of $\PPP$ that are contained in $\wt{I}_j$, then there
is an automorphism of $\C$ such that for each $v\in V_\Lam=A_\Lam$ we
have that 
$\pi_{Q_1}(v)=\rho\pa{\pi_{Q_2}(v)}$. From here~\eqref{ranks are equal} readily follows.

\end{proof}

\ignore{
We note a simple consequence:
\begin{prop}\name{prop: never decomposable}
Number field lattices are never decomposable. 
\end{prop}

\begin{proof}
A decomposable lattice $\Lambda = \Lambda_1 \oplus \Lambda_2$ has
nonzero vectors with zero entries (namely those nonzero vectors in
each $\Lambda_i$, as embedded in $\Lambda$). On the other hand for a
number field lattice the partition $\til \PPP$ is the trivial
partition into the single block $\{1, \ldots, n\}$. Thus if we had a
decomposable number field lattice, we would get a contradiction to the
case $t=1$ of Lemma \ref{rl}. 
\end{proof}
}

\subsection{Strict inequalities for $\kappa$ values} 

\begin{lem}\label{nei}
Let $H_1\Lam\subset H_2\Lam$ be  a containment of two closed orbits
where $H_i=H(\PPP_i)$, $i=1,2$. If the containment
$\A_\Lambda(\PPP_2)\subset\A_\Lambda(\PPP_1)$ is non-essential, then
there is a block of $\PPP_2$ that contains two distinct blocks of
$\PPP_1$ that are contained in the same block of $\wt{\PPP}_1$.  
\end{lem}
Lemma~\ref{nei} together with the following Theorem implies the
validity of Theorem~\ref{main theorem}.  
We postpone the proof of Lemma~\ref{nei} to the end of this section.
\begin{thm}\name{main theorem technical}
Let $H_1\Lam\subset H_2\Lam$ be  a containment of two closed orbits
where $H_i=H(\PPP_i)$, $i=1,2$. Suppose that there is a block of
$\PPP_2$ that contains two distinct blocks of $\PPP_1$ that are
contained in the same block of $\wt{\PPP}_1$. Then $\kappa_1 <\kappa_2.$
\end{thm} 
\begin{proof}
Assume by way of contradiction that $\kappa_1=\kappa_2$. 
Without loss of generality we may assume that $\Lam=\Lam_{\on{max}}$
is a lattice for which the conclusions of Theorem~\ref{thm:
  ergodic+semicontinuous} are satisfied for the $A$-invariant set
$H_1\Lam$; 
that is, we assume that $\kappa(\Lam)=\max\set{\kappa(\Lam'):\Lam'\in
  H_1\Lam}$ and furthermore, that the symmetric cube $\mathcal{C}$ of
volume $2^n \kappa_1$  
is admissible for $\Lam$. From our assumption we deduce that also
$\kappa(\Lam)=\max\set{\kappa(\Lam'):\Lam'\in H_2\Lam}$. In practice,
the property of $\Lam$ 
that will be of importance to us is that one cannot act on $\Lam$ with
some $h\in H_2$ in such a way that a symmetric box of bigger volume
will be admissible for $h\Lam$.  

For $1\le i\le n$ we refer to the face of $\mathcal{C}$ that intersects
the positive $i$-th axis as the $i$-th face of $\mathcal{C}$ and
denote it by $\mathcal{F}_i$. By symmetry there will be no need to consider
the opposite faces. 
We divide the rest of the argument into steps.

\medskip

\noindent\textbf{Step} 1. 
We first observe that for each $1\le i\le n$
the relative interior of $\mathcal{F}_i$ intersects $\Lam$
nontrivially. If not, we could have chosen an admissible
symmetric box for $\Lam$ which strictly contains $\mathcal{C}$ and in
particular, is of greater volume. We refer to the set $L\df
\Lam\cap\partial \mathcal{C}$ as the set of \textit{locking points}
and denote  
by $L_i$, $1\le i\le n$ the set of points in $L$ that belong to the
relative interior of $\mathcal{F}_i.$  

\medskip

\noindent\textbf{Step} 2. 
Recall that for $Q\subset \set{1,\dots, n}, \, \pi_Q:\R^n\to \R^{|Q|}$
is the projection to the coordinates of $Q$. 
We observe that for each block $Q$ of the partition $\PPP_2$ and for each $i_0\in Q$,
\begin{equation}\label{in the conv}
0\in\on{conv}\set{\pi_{Q\smallsetminus\set{i_0}}(v):v\in L_{i_0}}.
\end{equation}
In other words, 
if we restrict attention to the coordinates of the block $Q$, then the
point of intersection of $\mathcal{F}_{i_0}$ with the 
 $i_0$-th axis  belongs to the convex hull of locking points for the
 $i_0$-th face. 

To prove~\eqref{in the conv}, suppose to the contrary that $0$ does
not belong to $\on{conv}\set{\pi_{Q\smallsetminus\set{i_0}}(v):v\in L_{i_0}}$. Then
there is a linear functional $f: \R^{|Q|-1} \to \R$ which is strictly
positive on 
$\set{\pi_{Q\smallsetminus\set{i_0}}(v):v\in
  L_{i_0}}$. Define a one-parameter unipotent subgroup $\{u_t\}$ of
$G$ by the formula  
\eq{eq: ut}{
u_t(v) = v+tf(\pi_{Q\smallsetminus \set{i_0}}(v))\E_{i_0},
}
where $\E_1, \ldots, \E_n$ is the standard basis of $\R^n$. 
Since $Q$ is a block of the partition defining $H_2$, we have
$\set{u_t}\subset H_2$. It follows from the definition of $u_t$ and
the positivity of  
$f\circ\pi_{Q\smallsetminus\set{i_0}}$ on $L_{i_0}$ that
for small values of $t>0$  the cube $\mathcal{C}$ is admissible for
$u_t\Lam$. Furthermore, $u_t\Lam$ does not contain points in the
relative interior of $\mathcal{F}_{i_0}$. Thus we may find
a symmetric box that strictly contains $\mathcal{C}$ and is
admissible for $u_t\Lam$, contradicting our assumption that  
$\kappa$ attains its maximal value on $H_2\Lam$ at  $\Lam$.

\medskip

\noindent\textbf{Step} 3. Now we apply the Kernel
Lemma~\ref{rl}. Denote  by $Q$ a block of $\PPP_2$ that contains two  
distinct blocks $Q_1,Q_2$ of $\PPP_1$ such that both $Q_i$ are
contained in the same block of $\wt{\PPP}_1$, the existence of which
is assumed in the statement. Let 
$i_0\in Q_1$. 

Let $\delta>0$ be such that $\delta\E_{i_0}$ is the point of intersection
of $\mathcal{F}_{i_0}$ with the $i_0$-th axis. Because all the points
of $\mathcal{F}_{i_0}$ share  
the  same $i_0$-th coordinate, namely $\delta$, we deduce from \eqref{in
  the conv} that $\pi_Q(\delta\E_{i_0})$ is in
the convex hull of $\pi_Q(L_{i_0})$. 
Choose $v_1, \ldots, v_t$ in $L_{i_0}$ 
and a coefficients vector $\vec{\lambda}=(\lambda_1,\dots,\lambda_t)$ with $\lambda_i\ge 0$ and $\sum_1^t\lambda_i=1$ so that $\sum \lambda_i\pi_Q(v_i)=\pi_Q(\delta\E_{i_0}).$
In particular, as $i_0\notin Q_2$, the vector $\vec{\lambda}$ belongs to the kernel of the $n\times t$ matrix whose columns are the $\pi_{Q_2}(v_i)$'s. Applying the Kernel Lemma~\ref{rl} we deduce that
there is an automorphism $\rho$ of $\C$ so that $\rho(\vec{\lambda})$ belongs to the kernel of the corresponding matrix with columns $\pi_{Q_1}(v_i)$. As $\rho$ is an automorphism we deduce
that $\sum_1^t\rho(\lambda_i)=\rho(1)=1$. Looking at the $i_0$ coordinate we deduce that $0=\sum_1^t\rho(\lambda_i)\delta=\delta$, a contradiction.

\ignore{
 and such that \equ{eq: in conv1'} cannot be satisfied with fewer than $t$ elements of $L_{i_0}$. 
Since $i_0\notin Q_2$, \eqref{in the conv} implies that $0 \in
\conv\set{\pi_{Q_2}(v_i): 1\le i\le t}$. In particular,  
the projections $\set{\pi_{Q_2}(v_i)}_{i=1}^t$ are linearly
dependent. Lemma~\ref{rl} implies 
that the vectors $\set{\pi_{Q_1}(v_i)}_{i=1}^t$ are linearly
dependent as well. Let $\alpha_1, \ldots, \alpha_t$ be scalars, not
all zero, such that $\sum \alpha_i \pi_{Q_1}(v_i)=0$. Since the
$i_0$-coordinate of  all of the $v_i$ is equal to $\rho$ and $i_0\in
Q_1$ we find 
\eq{eq: sum alphas'}{
\sum_1^t \alpha_i =0.
}
By~\eqref{eq: in conv1'} we may choose  $\lambda_i > 0$ with $\sum
\lambda_i=1$  such that  $\sum_1^t \lambda_i \pi_Q(v_i)=\pi_Q(\rho
\E_{i_0})$.  
Then for each $s \in \R$ we have
\eq{eq: contradicts'}{\sum b_i(s)\pi_Q(v_i)=\pi_Q(\rho\E_{i_0}), \ \ \
  \ \mathrm{where} \ \ b_i(s)\df \lambda_i+s\alpha_i.} 
By \equ{eq: sum alphas'} at
least one of the $\alpha_i$ is negative. Therefore there is some $i$ such that  
$b_i(s) \longrightarrow_{s\to\infty} -\infty$, so there is a largest
positive $s$ so that all the coefficients $b_i(s)$ are
non-negative. By continuity at least one of these coefficients
vanishes, but not all since $\sum b_i(s) \equiv 1.$ So \equ{eq:
  contradicts'} gives us a representation of $\pi_Q(\rho\E_{i_0})$ as
a convex combination of fewer than $t$ elements of $\pi_Q(L_{i_0})$
--- a contradiction to the minimality of $t$.  
}
\end{proof}
\begin{proof}[Proof of Lemma~\ref{nei}]
The associated algebras $\A_\Lambda(\PPP_1),$ $\A_\Lambda(\PPP_2)$ are
isomorphic to direct sums of number fields.  
Suppose $\A_\Lambda(\PPP_i)\cong \oplus_{j=1}^{r_i} F_j^{(i)}$. Then
the inclusion $\A_\Lambda(\PPP_2)\subset \A_\Lambda(\PPP_1)$ induces
an embedding  
$$\varphi:\oplus_{j=1}^{r_2} F_j^{(2)}\hookrightarrow
\oplus_{j=1}^{r_1} F_j^{(1)}.$$ For such an embedding, there exists a
partition $\set{1,\dots,r_1}=\bigsqcup_{j=1}^{r_2} Q_j$ 
such that for each  $ 1\le j\le  r_2$  there is an embedding  
$\varphi_j: F_j^{(2)}\hookrightarrow \oplus_{i\in Q_j}F_i^{(1)}$ so
that $\varphi$ takes the following form  
$$\varphi( (x_j)_{j=1}^{r_2}) =
\pa{\varphi_1(x_1)),\dots,\varphi_{r_2}(x_{r_2}) }\in \pa{\oplus_{i\in
    Q_1} F_i^{(1)}}\oplus\dots\oplus\pa{\oplus_{i\in Q_{r_2}}
  F_i^{(1)}}.$$ 
Our assumption that the embedding $\varphi$ is non-essential simply
means that there exists some $1\le j_0\le r_2$ and $i_0\in Q_{j_0}$
such that  
\begin{equation}\label{gt1}
\deg\left(F_{i_0}^{(1)}/\Q \right ) > \deg \left(F_{j_0}^{(2)}/\Q \right
).
\end{equation}
Recall that by
Theorem~\ref{prop: algebra} the blocks of the partition $\PPP_2$
correspond to the various embeddings of the $F_j^{(2)}$. Choose a
block $Q$ that 
corresponds to one of the embeddings of $F_{j_0}^{(2)}$. 

This block splits into several blocks of the finer
partition $\PPP_1$ and this splitting is done in two steps. First, it 
splits into blocks $S_i, i\in Q_{j_0},$ that correspond to the
embedding $\varphi_{j_0}:F^{(2)}_{j_0}\hookrightarrow\oplus_{i\in
  Q_{j_0}} F_i^{(1)}$ and then each block $S_i$ further splits into
$\deg \left(F^{(1)}_{j_0}/\Q \right) / \deg \left(F^{(2)}_i/\Q\right)$ blocks $S_{i\ell}$ which
are the blocks of $\PPP_1$. 
Following the definitions we see that for each fixed $i$ the blocks
$S_{i\ell}$ belong to the same block of $\wt{\PPP}_1$; namely the one
defined by $F_i^{(1)}$.  
Combining this with the inequality~\eqref{gt1} concludes the proof.
\end{proof}
\section{Proofs of isolation results}\name{sec: completion}
We recall the main dynamical result of \cite{LW}, which is the basis
of our proof of isolation results.  
\begin{thm}\name{thm: dynamical isolation}
Let $n \geq 3$, let $H\Lambda$ be a finite-volume homogeneous subspace of $\Xn$,
corresponding via Corollary \ref{cor: 
  bijective} to a subfield $F$ of the associated algebra
$\A_\Lambda$. Assume that $\overline{A\Lambda} = H\Lambda$ and let
$(\Lambda_k)$ be a sequence of lattices in $\Xn \sm H\Lambda$
converging to $\Lambda$. Then, after passing to a subsequence of the
$(\Lambda_k)$, there is a proper subfield $K \varsubsetneq F$ such
that the set of accumulation points of the form $\lim a_k \Lambda_k$, $a_k\in A$,
is equal to the finite-volume homogeneous space $H'\Lambda$, where
$H'$ is an equiblock group 
associated to $K$ via Corollary \ref{cor: bijective}.  

In particular, if $F$ has no proper subfields other than $\Q$, then
the set of accumulation points as above is the entire space $\Xn$.

\end{thm}
Although the Theorem is not stated in this form in \cite{LW}, this
statement follows from the proof of  \cite[Theorem 1.3]{LW}. For more
details see \cite[Corollary 7.1]{LW} and \cite[Theorem 4.8]{ELMV}. 
\subsection{Proofs}

\begin{proof}[Proof of Theorem \ref{thm: isolation nf}]
For each subfield $F_1\subset F$, by Corollary \ref{cor: bijective}
there is a corresponding homogeneous subset $H\Lambda$, equipped with
a homogeneous measure $\mu_{F_1}$, and by Theorem \ref{thm:
  ergodic+semicontinuous}, a corresponding 
$\kappa_{\mu_{F_1}}$. Let  
$$\kappa' \df  \min\{\kappa_{\mu_{F_1}}: F_1 \varsubsetneq F \mathrm{\
  a \ subfield} \}.$$ 
Note that if $F$ has no proper subfields, the only possible $F_1$ is
the field $\Q$, and in this case $\kappa'=1$. Also note that by
Theorem \ref{u thm 2}, $\kappa'> \kappa(\Lambda)$.  

We now claim that 
for any sequence $\Lambda_k \to \Lambda$, such that 
$\Lambda_k \notin A\Lambda$, we have  $\liminf_k \kappa(\Lambda_k)
\geq \kappa'$. Take a subsequence along which
$\kappa(\Lambda_k)$ converges.  
 Applying Theorem \ref{thm: dynamical isolation} in the special case
 $H=A$, after passing to a further subsequence we find that there is a
 subfield $F_1 \varsubsetneq F$ such that any lattice in $H'\Lambda$
 is an accumulation point of a sequence of the form $a_k
 \Lambda_k$. Here $H'$ is the equiblock group corresponding to $F_1$
 under Corollary \ref{cor: bijective}. 
In particular, we can choose $\Lambda_{\max}$, a lattice realizing the
maximal value of $\kappa$ on the homogeneous subset $H'\Lambda$ as in
Theorem \ref{thm: ergodic+semicontinuous}, as a limit point of
$a_k\Lambda_k$. In view of Proposition \ref{prop: standard},  
$$\kappa' \leq \kappa(\Lambda_{\max}) \leq \lim\kappa(a_k\Lambda_k) =
\lim \kappa(\Lambda_k).$$

Now to prove local isolation, note that Definition \ref{def:
  isolated} is satisfied with $\vre_0 \df
\kappa'-\kappa(\Lambda)$. This also implies strong isolation when $F$
has no proper subfields. It remains to show that $\Lambda$ is not
strongly isolated when $F$ does have a proper subfield $F'$. Indeed,
in this case 
$\kappa'<1$ by Theorem \ref{u thm 2}. Letting $H'$ denote the block group corresponding
to $F'$, we find from Theorem \ref{thm: ergodic+semicontinuous} that
there is a dense collection of lattices $\Lambda'
\in H'\Lambda$, for which $\kappa(\Lambda')=\kappa'$. This means that
$\Lambda$ is not strongly isolated. 
\end{proof} 

\begin{proof}[Proof of Corollary \ref{cor: extends}]
We first recall that for any $n$ there is a totally real number field $F$ of degree $n$ without proper subfields. Indeed, by \cite[Prop. 2]{field extensions}, for any $n$ there is a totally real Galois extension $K$ of $\Q$ with Galois group $\mathcal{S}_n$ (the full permutation group of $\{1, \ldots, n\}$). The subgroup $G_0 \cong \mathcal{S}_{n-1}$ fixing the element 1 is a maximal subgroup of index $n$, so by the Galois correspondence, the subfield $F$ of $K$ fixed by $G_0$ has the required properties. 

Now let $\Lambda$ be a number field lattice 
in $\R^n$ arising from such a field $F$ via the construction as in
\eqref{eq: vectors first type} of \S\ref{subsec: constructions}. By
Theorem \ref{thm: isolation nf}, $\Lambda$ is strongly isolated, and
by Proposition \ref{prop: density} (applied to $H_1=A,H_2=G$), the set
of number field lattices with associated field $F$ is dense in $\Xn$.  
\end{proof}

We now give a variant of Definition \ref{def: isolated}.
Let $\Lambda\in \Xn$ and let $H \subset G$ be a subgroup containing
$A$. Given $\vre_0>0$, we say that $\Lambda$ is {\em $\vre_0$-isolated
  relative to $H$} if for any $0<\vre<\vre_0$ there is a neighborhood
$\mathcal{U}$ of $\Lambda$ in $\Xn$, so that for any $\Lambda' \in
\mathcal{U} \sm H\Lambda$, $\kappa(\Lambda') >\kappa(\Lambda)+\vre$.  
We will say that $\Lambda$ is 
{\em  locally isolated relative to $H$} if it is  $\vre_0$-isolated relative to $H$ for some $\vre_0>0$. 

With this definition we prove the following result, of which Theorem
\ref{thm: relative isolation} is a special case: 

\begin{thm}
Let $\mu$ be a homogeneous $A$-invariant probability measure which
corresponds to the homogeneous space $H\Lambda_0$ with $A
\varsubsetneq H \varsubsetneq  G$. 
Then for any $\Lambda\in H\Lambda_0$ for which $\overline{A\Lambda} =
H\Lambda_0$ (in particular, for $\mu$ -almost any $\Lambda$) the
following assertions hold:  
\begin{enumerate}
\item
$\Lambda$ is not locally isolated. 
\item
$\Lambda$ is $\vre_0$-locally isolated relative to $H$, for 
$$
\vre_0 \df \min\{\kappa_\nu: \displaystyle^{\nu \textrm{ is a
    homogeneous $A$-invariant probability}}_{\textrm{measure with }
  \operatorname{supp}(\mu) \varsubsetneq \operatorname{supp}(\nu)} 
  \} - \kappa(\Lambda).
$$
\end{enumerate}
\end{thm}
Since there are only finitely many equiblock groups that
contain $H$, there are only finitely many measures $\nu$ that can
appear in the minimum defining 
$\vre_0$ above. By Theorem~\ref{u thm 2} we see that indeed $\vre_0>0$.

\begin{proof}
Since $A \varsubsetneq H$,  (1) is immediate from Theorem \ref{thm:
  ergodic+semicontinuous}, taking a sequence of generic elements in
$H\Lambda \sm A\Lambda$ tending to $\Lambda$. The proof of assertion
(2) is identical to the proof of Theorem \ref{thm: isolation nf},
except that in applying Theorem \ref{thm: dynamical isolation}, we use
$H$ in place of $A$.    
\end{proof}
\begin{proof}[Proof of Theorem \ref{thm: non-isolated, intermediate}]
By Proposition~\ref{pilc} there are number
field lattices in $H\Lambda$, and by Proposition \ref{prop: density} the
collection of number field lattices in $H\Lambda$ is dense. Let
$\Lambda_0\in H\Lambda$ be a lattice realizing the generic value
$\kappa_\mu$ and choose $\Lambda_k\to\Lambda_0$ a sequence of number
field lattices from within $H\Lambda$. On the one hand, by
Theorems~\ref{thm: ergodic+semicontinuous} and \ref{u thm 2} we know that
$\kappa(\Lambda_k)<\kappa_\mu$. On the other hand, by Proposition
\ref{prop: standard},  
$\liminf_k\kappa(\Lambda_k)\ge \kappa_\mu$. It follows that the
sequence $\kappa(\Lambda_k)$ converges to $\kappa_\mu$ and after
possibly taking a subsequence, we may assume  
that $\kappa(\Lambda_k) \nearrow \kappa_\mu$. Finally, by
Proposition~\ref{prop: never decomposable}, these values belong to
the reduced Mordel-Gruber spectrum. 
\end{proof}

\begin{proof}[Proof of Theorem \ref{thm: hopefully2}]
Given $t$, let $\Q \varsubsetneq F_1 \varsubsetneq \cdots \varsubsetneq F_t$ be a tower of totally real fields, and let $n=\deg( F_t/\Q)$. Let $\Lambda \subset \Xn$ be a number field lattice 
corresponding to $F_t$, constructed via \equ{eq: vectors first type}
for some rank $n$ subgroup $L \subset F_t$. Then each of the $F_i$ is
obtained as $\A_{\Lambda}(\PPP_i)$ for some partition $\PPP_i$, and by
Corollary \ref{cor: bijective}, 
the corresponding groups $H_i \df H(\PPP_i)$ satisfy $A=H_t
\varsubsetneq H_{t-1} \varsubsetneq \cdots \varsubsetneq H_1
\varsubsetneq G$.  
For each $i$, let $H_i\Lambda$ be the corresponding finite-volume homogeneous
subspace. 
Denote $\SL(V_\Lambda)\df G\cap \GL(V_\Lambda)$, i.e., the set of
elements of $G$ which are rational with respect to the $\Q$-structure
induced by $\Lam$.
For each $q \in \SL(V_\Lambda)$ and each $i$, the orbit
$H_iq\Lambda$ is also a homogeneous subspace, since $q$ commensurates
$\on{Stab}_G(\Lambda)$. Let  $\kappa_i(q)$ denote the generic value of $\kappa$,
as in Theorem \ref{thm: ergodic+semicontinuous}, on the homogeneous
subspace $H_iq\Lambda$. We will show by induction that each
$\kappa_{t-i}(q)$ belongs to $\widehat{\mathbf{MG}}_n^{(i)}.$ 

Suppose first that $i=1$. Then each $H_{t-1}q\Lambda$ is a homogeneous
subspace, which contains the compact $A$-orbits $Aq'\Lambda$, for all
$q' \in \SL(V_\Lambda) \cap H_{t-1}q.$ Since $\SL(V_\Lambda)$ is a group and
$\SL(V_\Lambda) \cap H_i$ is dense in each $H_i$, the set of such $q'$ is
dense in each $H_{t-1}q$. Therefore, repeating the argument proving
Theorem \ref{thm: non-isolated, intermediate}, we find that each
$\kappa_{t-1}(q)$ is a limit of an increasing sequence from
$\widehat{\mathbf{MG}}_n^{(i)}.$ For the case of general $i$ we argue
in the same way, taking all $q'  
\in \SL(V_\Lambda) \cap H_{t-i+1}q$, and using the values of $\kappa$
corresponding to $H_{t-i}q'\Lambda$ to approximate the value
$H_{t-i+1}q\Lambda$.  
\end{proof}

\section{The case $n=2$}
\name{sec: n=2}
For a lattice $\Lambda \subset \R^n$, we denote 
$$\lambda(\Lambda) \df \inf \left\{\left|\prod x_i\right | : (x_1, \ldots, x_n) \in \Lambda \right\}.$$
The following was proved in \cite{Gruber}:
\begin{prop}[Gruber]\name{prop: Gruber}
For a lattice $\Lambda$ of dimension 2, $\kappa(\Lambda)<1 \Longleftrightarrow \lambda(\Lambda)>0.$

\end{prop}
\begin{remark}
Using the results of the previous sections, it is not hard to show that Gruber's result is not valid for general $n$. Indeed, Mahler's compactness criterion (Proposition \ref{prop: Mahler}) implies that the condition $\lambda(\Lambda)>0$ is equivalent to the boundedness of the $A$-orbit of $\Lambda$ in $\Xn$. Now let $n$ be composite and let $\mu$ be a homogeneous measure on $\Xn$ supported on intermediate 
lattices which are not 
number field lattices. In light of \cite[Step 6.3]{LW} and Proposition \ref{prop: ergodic}, for almost any $\Lambda \in \on{supp}(\mu)$, $A\Lambda$ is not bounded, so that $\lambda(\Lambda)=0$. However by Theorem \ref{u thm 2}, $\kappa(\Lambda)<1$. 
\end{remark}

\begin{proof}[Proof of Theorem \ref{thm: n=2}]
Let $\Lambda$ be a lattice in dimension 2, with
$\kappa(\Lambda)<1$. We wish to show that it is not strongly
isolated. In view of Proposition \ref{prop: Gruber}, we know that
$\lambda(\Lambda)>0$, and it suffices to show that there is a bounded
$A$-orbit $A\Lambda_0$ which contains $A \Lambda$ in its closure but
is not equal to $A\Lambda$. Since $n=2$, the $A$-action is the
geodesic flow on the unit tangent bundle to the modular surface, and
the existence of such orbits is well-known using symbolic
dynamics. More specifically, using the viewpoint of
\cite{AdlerFlatto}, for any lattice $\Lambda \in \mathcal{L}_2$, let
$\alpha, \omega$ be two real numbers which are endpoints of the
infinite geodesic through a lift of the tangent vector corresponding
to $\Lambda$ in the upper half-plane. Since $A\Lambda$ is bounded, the
continued fractions coefficients of the numbers $\alpha, \omega$ are
bounded, say by a number $k$. Denote these coefficients by $\alpha =
[a_{-1}, a_{-2}, \ldots ]$ and $\omega = [a_0, a_1, a_2, \ldots].$ Now
let $\alpha' \df [k+1, k+1, \ldots]$ and  
$$
\omega' \df [a_0, a_{-1}, a_0, a_1, a_{-2}, a_{-1}, a_0, a_1, a_2, a_{-3}, a_{-2}, \ldots].
$$
That is, the bi-infinite word obtained by concatenating the expansions
of $\alpha$ and $\omega$ is in the orbit-closure, under the shift, of
the the bi-infinite word obtained by concatenating the expansions of
$\alpha', \omega'$.  

In view of the symbolic coding of the geodesic flow
\cite{AdlerFlatto}, the closure of the projection in $\mathcal{L}_2$ of the geodesic with endpoints $\alpha', \omega'$
contains the projection of the geodesic with endpoints $\alpha, \omega$. 
Since the digits of $\alpha'$ are greater than $k$, the two
orbits are distinct. Since all digits of $\alpha', \omega'$ are
bounded by $k+1$, the corresponding $A$-orbit has $\kappa <1$.   
\end{proof}

Theorem \ref{thm: non-isolated, intermediate} concerns with the existence
of accumulation points for $\widehat{\mathbf{MG}}_n$ besides 1, for $n
\geq 3$. As we now explain, Proposition \ref{prop: Gruber} can be used
to settle this question in dimension 2.  

\begin{prop}\name{prop: accumulation n=2}
The set $\mathbf{MG}_2 = \widehat{\mathbf{MG}}_2$ has accumulation points smaller than 1. 
\end{prop} 

\begin{proof}
Let $\mathbf{Bad}_k$ denote the set of real numbers $x$ whose
continued fraction coefficients $a_1(x), a_2(x), \ldots$ are bounded
above by $k$. Then it is well-known that $\cup_{k\ge 1}\mathbf{Bad}_k$
contains all real quadratic irrationals. Given a real quadratic
irrational $x$, let $L= \Z \oplus \Z x$ be an additive subgroup in the
corresponding quadratic field $\Q(x)$, and $\Lambda = \Lambda(x) \in
\mathcal{L}_2$ be the lattice in dimension 2, constructed via
\equ{eq: vectors first type}.  
Then, as is well-known (and is a very special case of Corollary \ref{r.21.6})
the orbit $A\Lambda(x)$ is compact. 
The inequalities of \cite{Gruber} imply that a uniform bound on the
continued fraction coefficients of $x$ imply a uniform bound on the
Mordell constant; in particular, for any $k$ there is $\kappa_0<1$ so
that if $x \in \mathrm{Bad}_k$ is a quadratic irrational, then
$\kappa(\Lambda(x)) \leq \kappa_0$.  

It is known that there is $k$ such that $\mathbf{Bad}_k$ contains a
sequence $(x_n)$ of quadratic irrationals, for which the fields $F_n
\df \Q(x_n)$ are distinct quadratic fields. Indeed, as explained to
the authors by Dmitry Kleinbock, one can take $k=2$. By \cite[Theorem
1.6]{Cusick}, the quadratics in $\mathbf{Bad}_2$ are the numbers
$\sqrt{(3m-2)(3m+2)}/m$, and it follows from Dirichlet's theorem on
primes in arithmetic progressions, that among these, numbers belonging
to infinitely many distinct fields arise. 

For each $n$, as in Step 1 of Theorem \ref{main theorem technical}, 
there are $v_1^{(n)}, v_2^{(n)}$ which
are locking points for $\Lambda(x_n)$. Let $\sigma^{(n)}_1,
\sigma^{(n)}_2$ be the two field embeddings of $F_n$. After applying
an element of $A$, we find by \equ{eq: vectors first type} that there
are $\alpha_n, \beta_n$ in $F_n$ such that $v^{(n)}_1 =
\left(\sigma^{(n)}_1(\alpha_n), \sigma^{(n)}_2(\alpha_n)\right)$ and
$v^{(n)}_2 =  \left(\sigma^{(n)}_1(\beta_n),
  \sigma^{(n)}_2(\beta_n)\right) $, so that $$\kappa_n \df \kappa
(\Lambda(x_n)) = \sigma^{(n)}_1(\alpha_n) \cdot
\sigma^{(n)}_2(\beta_n).$$ 
Since $\alpha_n, \beta_n$ span $F_n$, they are linearly independent
over $\Q$. On the other hand
$\sigma^{(n)}_1(\alpha_n)\sigma^{(n)}_2(\alpha_n) \in \Q$. This
implies that $\kappa_n$ is irrational. Since the $\kappa_n$ belong to
distinct quadratic fields, they are therefore distinct. So the
sequence $(\kappa_n)$ is an infinite sequence in $\mathbf{MG}_2$,
bounded above by $\kappa_0 < 1$. This implies that $\mathbf{MG}_2$ has
a limit point smaller than 1.  
\end{proof}

\begin{remark}
By a similar argument, in order to show that there are infinitely many
distinct accumulation points in $\mathbf{MG}_2$, it suffices to
construct infinitely many disjoint finite blocks of natural numbers
$B_n$, and for each $n$, an infinite sequence of quadratics in
distinct fields, whose continued fractions coefficients lie in
$B_n$. It would be interesting to know whether $\mathbf{MG}_2$ is
uncountable.   
\end{remark}

\end{document}